\documentclass[10pt,a4paper,english]{amsart}
\usepackage{babel}
\usepackage[numbers]{natbib}
\usepackage{amsfonts, amsmath, wasysym}
\usepackage{amssymb, amsthm}
\usepackage{mathrsfs}
\usepackage{verbatim}
\newfont{\ffont}{cmr10}

\allowdisplaybreaks

\newcommand{\cE}{\mathcal{E}}

\newcommand{\cM}{\mathcal{M}}
\newcommand{\cN}{\mathcal{N}}

\newcommand{\id}{\mathbf{1}}

\newcommand{\vNT}{\overline{\otimes}}

\newcommand{\ot}{\otimes}
\newcommand{\tr}{\tau}
\newcommand{\al}{\alpha}

\newcommand{\si}{\sigma}

\newcommand{\R}{\mathbb{R}}
\newcommand{\C}{\mathbb{C}}
\newcommand{\N}{\mathbb{N}}

\newcommand{\ti}{\times}

\newcommand{\epf}{}

\newcommand{\xspace}{\hbox{\kern-2.5pt}}

\renewcommand{\labelenumi}{(\alph{enumi})}

\begin{document}

\newtheorem{definition}{Definition}[section]
\newtheorem{theorem}[definition]{Theorem}
\newtheorem{conjecture}[definition]{Conjecture}
\newtheorem{proposition}[definition]{Proposition}
\newtheorem{corollary}[definition]{Corollary}
\newtheorem{remark}[definition]{Remark}
\newtheorem{lemma}[definition]{Lemma}
\newtheorem{example}{Example}[section]
\newtheorem{exercise}{Exercise}[section]

\title{Noncommutative Boyd interpolation theorems}
\author{Sjoerd Dirksen}

\address{Universit\"{a}t Bonn\\
Hausdorff Center for Mathematics\\
Endenicher Allee 60\\
53115 Bonn\\
Germany}
\email{sjoerd.dirksen@hcm.uni-bonn.de}

\thanks{This research was supported by VICI subsidy 639.033.604 of the Netherlands Organisation for Scientific Research (NWO) and the Hausdorff Center for Mathematics}
\maketitle

\begin{abstract}
We present a new, elementary proof of Boyd's interpolation theorem. Our approach naturally yields a noncommutative version of this result and even allows for the interpolation of certain operators on $l^1$-valued noncommutative symmetric spaces. By duality we may interpolate several well-known noncommutative maximal inequalities. In particular we obtain a version of Doob's maximal inequality and the dual Doob inequality for noncommutative symmetric spaces. We apply our results to prove the Burkholder-Davis-Gundy and Burkholder-Rosenthal inequalities for noncommutative martingales in these spaces.
\end{abstract}

\section{Introduction}

Symmetric Banach function spaces play a pivotal role in many fields of mathematical analysis, especially probability theory, interpolation theory and harmonic analysis.
A cornerstone result in the interpolation theory of these spaces is the Boyd interpolation theorem, named after D.W.\ Boyd. Together with the Calder\'{o}n-Mitjagin
theorem, which characterizes the symmetric Banach function spaces which are an exact interpolation space for the couple $(L^1(\R_+),L^{\infty}(\R_+))$, Boyd's theorem
provides an invaluable tool for the analysis of symmetric spaces.\par The history of Boyd's interpolation theorem begins with the announcement of Marcinkiewicz
\cite{Mar39}, shortly before his death, of an extension of the Riesz-Thorin theorem. Let us say that a sublinear operator $T$ is of \emph{Marcinkiewicz weak type}
$(p,p)$ if for any $f \in L^p(\R_+)$,
$$d(v;Tf)^{\frac{1}{p}} \leq C v^{-1}\|f\|_{L^p(\R_+)} \qquad (v>0),$$
where $d(\cdot;Tf)$ denotes the distribution function of $Tf$. Marcinkiewicz demonstrated that if a sublinear operator $T$ is simultaneously of Marcinkiewicz weak types
$(p,p)$ and $(q,q)$ for $1\leq p<q\leq\infty$, then $T$ is bounded on $L^r(\R_+)$, for any $p<r<q$. A full proof of this result was published years later by Zygmund
\cite{Zyg56}, based on Marcinkiewicz' notes. Soon after it was observed by Stein and Weiss \cite{StW59} that Marcinkiewicz' result is valid for the larger class of
operators which are simultaneously of \emph{weak types} $(p,p)$ and $(q,q)$. Here $T$ is said to be of weak type $(p,p)$ if for any $f$ in the Lorentz space
$L^{p,1}(\R_+)$,
$$d(v;Tf)^{\frac{1}{p}}\leq C v^{-1}\|f\|_{L^{p,1}(\R_+)}.$$
The class of sublinear operators which are simultaneously of weak types $(p,p)$ and $(q,q)$ was subsequently characterized by Calder\'{o}n \cite{Cal66} as consisting of
precisely those maps $T$ which satisfy
$$\mu_t(Tf)\leq C S_{p,q}(\mu(f))(t) \qquad (t>0),$$
where $\mu(f)$ denotes the decreasing rearrangement of $f$ and $S_{p,q}$ is a linear integral operator which is nowadays known as Calder\'{o}n's operator. Finally, in
\cite{Boy69} Boyd introduced two indices $p_E$ and $q_E$ for any symmetric Banach function space $E$ on $\R_+$ and showed that the operator $S_{p,q}$ is bounded on $E$
precisely when $p<p_E\leq q_E<q$. Together with Calder\'{o}n's characterization, this yields Boyd's interpolation theorem: every sublinear operator which is simultaneously of weak
types $(p,p)$ and $(q,q)$ is bounded on $E$ if and only if $p<p_E\leq q_E<q$.\par In this paper we are concerned with obtaining a generalization of Boyd's result to
noncommutative and, to a limited extent, also noncommutative vector-valued symmetric Banach function spaces. As it turns out, the original approach sketched above remains feasible in the noncommutative setting (see the appendix of this paper), but becomes problematic for noncommutative vector-valued spaces. We develop a new, elementary approach to Boyd's interpolation theorem for the class of Marcinkiewicz weak type operators. Our approach consists of two observations, which are close in spirit to the original approach. Firstly, we characterize the sublinear operators of simultaneous Marcinkiewicz weak types $(p,p)$ and $(q,q)$ as being exactly those which for some $\alpha>0$ satisfy the inequality
\begin{equation*}
\displaystyle d(\alpha v;Tf)\leq d(v;\Theta_{p,q}f) \qquad (v>0),
\end{equation*}
where $\Theta_{p,q}$ is the linear operator defined by
$$\Theta_{p,q}f(s,t)=f(s)(t^{-\frac{1}{q}}\chi_{(0,1)}(t) + t^{-\frac{1}{p}}\chi_{[1,\infty)}(t)).$$
Secondly, we show that $\Theta_{p,q}$ is bounded on $E$ if $p<p_E\leq q_E<q$. Our approach immediately extends to yield both a vector-valued and a noncommutative version
of Boyd's result. Moreover, all results are valid for symmetric \textit{quasi}-Banach function spaces. Thus we obtain Boyd's theorem and its extensions for the full
scale of $L^p$-spaces.\par
Interestingly, our method even yields interpolation results for certain operators defined on noncommutative $l^1$- and $l^2$-valued
$L^p$-spaces in the sense of Pisier \cite{Pis98}. In particular, it allows for the interpolation of noncommutative probabilistic inequalities such as the dual Doob
inequality, in the noncommutative setting due to Junge \cite{Jun02}, and the `upper' noncommutative Khintchine inequalities, originally due to Lust-Piquard
\cite{L-P86,LuP91}. In fact, our approach has its origins in the proof of the Khintchine inequalities for noncommutative symmetric spaces given in \cite{DPP11,DiR11},
which the author only later understood as Boyd-type interpolation results.\par
By adapting the duality argument in Junge's proof of the Doob maximal inequality for noncommutative $L^p$-spaces, we can dualize our noncommutative $l^1$-valued interpolation theorem to find an interpolation result for noncommutative maximal inequalities. In particular, we deduce a version of Doob's maximal inequality for a large class of noncommutative symmetric spaces. In the final section we utilize the latter inequality and its dual version to prove Burkholder-Davis-Gundy and Burkholder-Rosenthal inequalities, respectively, for noncommutative symmetric spaces. Our results
extend the Burkholder-Gundy and Rosenthal inequalities established in \cite{DPP11}, as well as the Burkholder-Rosenthal inequalities for noncommutative $L^p$-spaces and
Lorentz spaces obtained in \cite{JuX03} and \cite{Jia10}, respectively.\par During the writing of this manuscript we discovered that an interpolation result for
noncommutative $\Phi$-moment inequalities associated with Orlicz functions was proved recently in \cite{BeC10}. We discuss the connection of our work with this result
and in fact show that many of our interpolation results have a `$\Phi$-moment version'.\par The paper is organized so that the first part, up to the classical Boyd interpolation theorem, can be read without any knowledge of noncommutative analysis.

\section{Symmetric quasi-Banach function spaces}
\label{sec:QBFS}

In this preliminary section we introduce symmetric quasi-Banach function spaces and discuss their most important properties. The results presented below are all well known for Banach function spaces, but not easy to find for quasi-Banach function spaces. We shall need the following well-known result due to T.\ Aoki and S.\ Rolewicz, which states that every quasi-normed vector space can be equipped with an equivalent $r$-norm (see e.g.\ \cite{KPR84} for a proof).
\begin{theorem}
\label{thm:Aoki-RolewiczOrig}
(Aoki-Rolewicz) Let $X$ be a quasi-normed vector space. Then there is a $C>0$ and $0<r\leq 1$ such that for any $x_1,\ldots,x_n \in X$,
\begin{equation}
\label{eqn:Aoki-Rolewicz}
\Big\|\sum_{i=1}^n x_i\Big\|_X \leq C \Big(\sum_{i=1}^n \|x_i\|_X^r\Big)^{\frac{1}{r}}.
\end{equation}
\end{theorem}
\noindent Let $\tilde{S}(\R_+)$ be the linear space of all measurable, a.e.\ finite functions $f$ on $\R_+$. For any $f \in \tilde{S}(\R_+)$ we define its
\emph{distribution function} by
$$d(v;f) = \lambda(t \in \R_+ \ : \ |f(t)|>v) \qquad (v\geq 0),$$
where $\lambda$ denotes Lebesgue measure. Let $S(\R_+)$ be the subspace of all $f \in \tilde{S}(\R_+)$ such that $d(v;f)<\infty$ for some $v>0$ and let $S_0(\R_+)$ be
the subspace of all $f \in S(\R_+)$ with $d(v;f)<\infty$ for all $v>0$. For $f \in S(\R_+)$ we denote by $\mu(f)$ the \emph{decreasing rearrangement} of $f$, defined by
\begin{equation*}
\label{eqn:decreasingRearrangement}
\mu_t(f) = \inf\{v>0 \ : \ d(v;f)\leq t\} \qquad (t\geq 0).
\end{equation*}
For $f,g \in S(\R_+)$ we say $f$ is \textit{submajorized} by $g$, and write $f\prec\prec g$, if
$$\int_0^t \mu_s(f) ds \leq \int_0^t \mu_s(g) ds, \qquad \mathrm{for \ all} \ t>0.$$
A (quasi-)normed linear subspace $E$ of $S(\R_+)$ is called a \emph{(quasi-)Banach function space} on $\R_+$ if it is complete and if for $f \in S(\R_+)$ and $g \in E$
with $|f|\leq |g|$ we have $f \in E$ and $\|f\|_E\leq \|g\|_E$. A (quasi-)Banach function space $E$ on $\R_+$ is called \textit{symmetric} if for $f \in S(\R_+)$ and $g
\in E$ with $\mu(f)\leq \mu(g)$ we have $f \in E$ and $\|f\|_E\leq\|g\|_E$. It is called \textit{fully symmetric} if, in addition, for $f \in S(\R_+)$ and $g \in E$ with
$f\prec\prec g$ it follows that $f \in E$ and $\|f\|_E\leq\|g\|_E$.\par
A symmetric (quasi-)Banach function space is said to have a \textit{Fatou (quasi-)norm} if for every net $(f_{\beta})$ in $E$ and $f \in E$ satisfying $0\leq
f_{\beta}\uparrow f$ we have $\|f_{\beta}\|_E\uparrow \|f\|_E$. The space $E$ is said to have the \textit{Fatou property} if for every net $(f_{\beta})$ in $E$
satisfying $0\leq f_{\beta}\uparrow$ and $\sup_{\beta}\|f_{\beta}\|_E<\infty$ the supremum $f=\sup_{\beta} f_{\beta}$ exists in $E$ and $\|f_{\beta}\|_E\uparrow
\|f\|_E$. We say that $E$ has \textit{order continuous} (quasi-)norm if for every net $(f_{\beta})$ in $E$ such that $f_{\beta}\downarrow 0$ we have
$\|f_{\beta}\|_E\downarrow 0$. In the literature, a symmetric (quasi-)Banach function space is often called \emph{rearrangement invariant} if it has order continuous
(quasi-)norm or the Fatou property. We shall not use this terminology.\par Let us finally discuss some results specific for symmetric Banach function spaces. The
\emph{K\"{o}the dual}\index{K\"{o}the dual} of a symmetric Banach function space $E$ is the Banach function space $E^{\ti}$ given by
\begin{align*}
& E^{\ti} = \Big\{g \in S(\R_+) \ : \ \sup\Big\{\int_0^{\infty} |f(t)g(t)| \ dt \ : \ \|f\|_E\leq 1\Big\}<\infty\Big\};\\
& \|g\|_{E^{\times}} = \sup\Big\{\int_0^{\infty} |f(t)g(t)| \ dt \ : \ \|f\|_E\leq 1\Big\}, \qquad g \in E^{\ti}.
\end{align*}
The space $E^{\ti}$ is fully symmetric and has the Fatou property. It is isometrically isomorphic to a closed subspace of $E^*$ via the map
\begin{equation*}
\label{eqn:KotheDualDef}
g \mapsto L_g, \qquad L_g(f) = \int_0^{\infty} f(t)g(t) \ dt \ \ \ (f \in E).
\end{equation*}
A symmetric Banach function space on $\R_+$ has a Fatou norm if and only if $E$ embeds isometrically into its second K\"{o}the dual $E^{\ti\ti}=(E^{\ti})^{\ti}$. It has
the Fatou property if and only if $E=E^{\ti\ti}$ isometrically. It has order continuous norm if and only if it is separable, which is also equivalent to the statement
$E^*=E^{\times}$. Moreover, a symmetric Banach function space which is separable or has the Fatou property is automatically fully symmetric. For proofs of these facts
and more details we refer to \cite{BeS88,KPS82,LiT79}.

\subsection{Boyd indices}
\label{sec:BoydIndicesNewSec}

We now discuss the \textit{Boyd indices}, which were introduced by D.W. Boyd in \cite{Boy69}. For any $0<a<\infty$ we define the dilation operator $D_a$ on $S(\R_+)$ by
$$(D_a f)(s) = f(as) \qquad (s \in \R_+).$$
The following lemma is well known for symmetric Banach function spaces (cf.\ \cite{KPS82}).
\begin{lemma}
\label{lem:dilation}
Let $E$ be a symmetric quasi-Banach function space on $\R_+$. Then, for every $0<a<\infty$, $D_a$ defines a bounded linear operator on $E$. Moreover, $a\mapsto\|D_a\|$ is a decreasing, submultiplicative function on $\R_+$.
\end{lemma}
\begin{proof}
Since $\mu(f)$ is decreasing, we have for any $a\leq b$,
$$D_b\mu(f)(s) = \mu_{bs}(f) \leq \mu_{as}(f) = D_a\mu(f)(s).$$
Hence, if $D_a$ is bounded on $E$, then $D_b$ is bounded on $E$ as well and $\|D_b\|\leq\|D_a\|$. In particular, $\|D_a\|$ is bounded on $E$ if $a\geq 1$ and $\|D_a\|\leq 1$. Moreover, it suffices to show that $D_{\frac{1}{n}}$ is bounded on $E$ for every $n \in \N$.\par
Fix $n \in \mathbb{N}$, let $f \in E_+$ and let $f_i$, $1\leq i\leq n$, be mutually disjoint functions having the same distribution function as $f$. Then $D_{\frac{1}{n}}f$ and $\sum_{i=1}^n f_i$ have the same distribution function. Indeed,
\begin{eqnarray}
\label{eqn:idenDistDil1}
\lambda(t \in \R_+ \ : \ (D_{\frac{1}{n}}f)(t)>v) & = & n\lambda(t \in \R_+ \ : \ f(t)>v) \nonumber \\
& = & \sum_{i=1}^n \lambda(t\in \R_+ \ : \ f_i(t)>v) \nonumber \\
& = & \lambda\Big(t \in \R_+ \ : \ \sum_{i=1}^n f_i(t)>v\Big).
\end{eqnarray}
Since $E$ is symmetric it follows that $D_{\frac{1}{n}}f \in E$. Moreover, by Theorem~\ref{thm:Aoki-RolewiczOrig}, there exists some $c>0$ and $0<p\leq 1$ such that
\begin{equation}
\label{eqn:dilEstpNorm}
\|D_{\frac{1}{n}}f\|_E = \Big\|\sum_{i=1}^n f_i \Big\|_E \leq c\Big(\sum_{i=1}^n \|f_i\|_E^p\Big)^{\frac{1}{p}} = c n^{\frac{1}{p}} \|f\|_E.
\end{equation}
From the above it is clear that $a\mapsto\|D_a\|$ is decreasing and, since $D_{ab}=D_a D_b$ if $a\leq b$, submultiplicative.
\epf
\end{proof}
\noindent Define the \textit{lower Boyd index}\index{Boyd indices} $p_E$ of $E$ by
$$p_E = \lim_{s\rightarrow\infty}\frac{\log s}{\log \|D_{\frac{1}{s}}\|}$$
and the \textit{upper Boyd index} $q_E$ of $E$ by
$$q_E = \lim_{s\downarrow 0}\frac{\log s}{\log \|D_{\frac{1}{s}}\|}.$$
By the Aoki-Rolewicz theorem $E$ admits an equivalent $p$-norm for some $0<p\leq 1$ and, as observed in (\ref{eqn:dilEstpNorm}), for every $n\geq 1$ we have
$$\|D_{\frac{1}{n}}f\|_E \leq n^{\frac{1}{p}} \|f\|_E$$
and therefore $p\leq p_E$. In particular we have $0<p_E\leq q_E\leq\infty$ and if $E$ is a symmetric Banach function space then $1\leq p_E\leq q_E\leq \infty$. One may show that the Boyd indices can alternatively be expressed as
\begin{equation}
\label{eqn:BoydIndexLowerDef}
p_E = \sup\Big\{p>0: \ \exists c>0 \ \forall 0<a\leq 1 \ \|D_a f\|_E\leq c a^{-\frac{1}{p}}\|f\|_E\Big\}
\end{equation}
and
$$q_E = \inf\Big\{q>0: \ \exists c>0 \ \forall a\geq 1 \ \|D_a f\|_E\leq c a^{-\frac{1}{q}}\|f\|_E\Big\}.$$
We shall need the following duality for Boyd indices (see \cite{KPS82}, Theorem II.4.11). If $E$ is a symmetric Banach function space with Fatou norm, then
\begin{equation}
\label{eqn:dualityBoydIndices}
\frac{1}{p_E} + \frac{1}{q_{E^{\ti}}} = 1, \qquad \frac{1}{p_{E^{\ti}}} + \frac{1}{q_E} = 1.
\end{equation}

\subsection{Convexity and concavity}
\label{sec:convexityConcavity}

Let $0<p,q\leq\infty$. A symmetric quasi-Banach function space $E$ is said to be $p$\textit{-convex}\index{quasi-Banach function space!$p$-convex} if there exists a constant $C>0$ such that for any finite sequence $(f_i)_{i=1}^n$ in $E$ we have
$$\Big\|\Big(\sum_{i=1}^n |f_i|^p\Big)^{\frac{1}{p}}\Big\|_E \leq C \Big(\sum_{i=1}^n\|f_i\|_E^{p}\Big)^{\frac{1}{p}} \qquad (\mathrm{if} \ 0<p<\infty),$$
or,
$$\Big\|\max_{1\leq i\leq n} |f_i| \ \Big\|_E \leq C \max_{1\leq i\leq n} \|f_i\|_E \qquad (\mathrm{if} \ p=\infty).$$
The least constant $M^{(p)}$ for which this inequality holds is called the $p$\textit{-convexity constant} of $E$.\\
A symmetric quasi-Banach function space $E$ is said to be $q$\textit{-concave}\index{quasi-Banach function space!$q$-concave} if there exists a constant $C>0$ such that for any finite sequence $(f_i)_{i=1}^n$ in $E$ we have
$$\Big(\sum_{i=1}^n\|f_i\|_E^{q}\Big)^{\frac{1}{q}} \leq C \Big\|\Big(\sum_{i=1}^n |f_i|^q\Big)^{\frac{1}{q}}\Big\|_E \qquad (\mathrm{if} \ 0<q<\infty),$$
or,
$$\max_{1\leq i\leq n} \|f_i\|_E \leq C \Big\|\max_{1\leq i\leq n} |f_i| \ \Big\|_E \qquad (\mathrm{if} \ q=\infty).$$
The least constant $M_{(q)}$ for which this inequality holds is called the $q$\textit{-concavity constant} of $E$. It is clear that every quasi-Banach function space is $\infty$-concave with $M_{(\infty)}=1$ and any Banach function space is $1$-convex with $M^{(1)}=1$.\par
For $1\leq r<\infty$, let the $r$\textit{-concavification}\index{quasi-Banach function space!concavification of} and $r$\textit{-convexification}\index{quasi-Banach function space!convexification of} of $E$ be defined by
\begin{eqnarray*}
\label{eqn:convexification}
E_{(r)} & := & \{g \in S(\R_+): \ |g|^{\frac{1}{r}} \in E\}, \ \|g\|_{E_{(r)}} = \| \ |g|^{\frac{1}{r}}\|_{E}^{r},\\
E^{(r)} & := & \{g \in S(\R_+): \ |g|^r \in E\}, \ \|g\|_{E^{(r)}} = \|\ |g|^r\|_{E}^{\frac{1}{r}},
\end{eqnarray*}
respectively. As is shown in \cite{LiT79} (p. 53), if $E$ is a Banach function space, then $E^{(r)}$ is a Banach function space. In general, $E_{(r)}$ is only a quasi-Banach function space. Using that $\mu(|f|^s)=\mu(f)^s$ for any $f \in S(\R_+)$ and $0<s<\infty$, one sees that $E^{(r)}$ and $E_{(r)}$ are symmetric if $E$ is symmetric. From the definitions one easily shows that if $E$ is $p$-convex and $q$-concave for $0<p\leq q\leq\infty$, then $E^{(r)}$ is $pr$-convex and $qr$-concave and $E_{(r)}$ is $\frac{p}{r}$-convex and $\frac{q}{r}$-concave. It is also clear from the definitions that
$$p_{E_{(r)}} = \frac{1}{r}p_E, \ q_{E_{(r)}}=\frac{1}{r}q_E, \ p_{E^{(r)}} = r p_E, \ q_{E^{(r)}}=r q_E.$$
We conclude this section by discussing two concrete classes of symmetric quasi-Banach function spaces in more detail.
\begin{example}
\label{exa:Lorentzpq} \textnormal{(Lorentz spaces $L^{p,q}$) Let $0<p,q\leq\infty$. The \textit{Lorentz space} $L^{p,q}$ is the subspace of all $f$ in $S(\R_+)$ such
that
\begin{align*}
\|f\|_{L^{p,q}} = \left\{\begin{array}{rl}(\int_0^{\infty}t^{\frac{q}{p}-1}\mu_t(f)^q \ dt)^{\frac{1}{q}} \qquad & (0<q<\infty),\\
\sup_{0<t<\infty} t^{\frac{1}{p}}\mu_t(f) \qquad & (q=\infty),
\end{array}
\right.
\end{align*}
is finite. If $1\leq q\leq p<\infty$ or $p=q=\infty$, then $L^{p,q}$ is a fully symmetric Banach function space. If $1<p<\infty$ and $p\leq q\leq \infty$ then $L^{p,q}$ can be equivalently renormed to become a fully symmetric Banach function space (\cite{BeS88}, Theorem 4.6). However, in general $L^{p,q}$ is only a symmetric quasi-Banach function space \cite{Kal84}. By the monotone convergence theorem, $L^{p,q}$ has the Fatou property. Its Boyd indices are determined by the first exponent, $p_{L^{p,q}} =
q_{L^{p,q}} = p$. The Lorentz space $L^{p,p}$ coincides with the Lebesgue space $L^p$. The spaces $L^{p,\infty}$ are referred to as \textit{weak $L^p$-spaces}.}
\end{example}
\begin{example}
\label{exa:Orlicz} \textnormal{(Orlicz spaces)\index{Orlicz space} Let $\Phi:[0,\infty)\rightarrow [0,\infty]$ be a Young\textquoteright s function, i.e., a convex,
continuous and increasing function satisfying $\Phi(0)=0$ and $\lim_{t\rightarrow \infty}\Phi(t) = \infty$. The \emph{Orlicz space} $L_{\Phi}$ is the subspace of all $f$
in $S(\R_+)$ such that for some $k>0$,
$$\int_0^{\infty} \Phi\Big(\frac{|f(t)|}{k}\Big) dt<\infty.$$
If we equip $L_{\Phi}$ with the Luxemburg norm
$$\|f\|_{L_{\Phi}} = \inf\Big\{k>0 \ : \ \int_0^{\infty} \Phi\Big(\frac{|f(t)|}{k}\Big) dt\leq 1\Big\},$$
then $L_{\Phi}$ is a symmetric Banach function space with the Fatou property \cite{BeS88,LiT79}.
The Boyd indices of $L_{\Phi}$ can be computed in terms of $\Phi$.
Indeed, let
$$M_{\Phi}(t)=\sup_{s>0}\frac {\Phi(ts)}{\Phi(s)},$$
and define the Matuszewska-Orlicz indices by
$$p_\Phi=\lim_{t\downarrow 0} \frac {\log M_{\Phi}(t)}{\log t},\qquad
q_\Phi=\lim_{t\to \infty} \frac {\log M_{\Phi}(t)}{\log t}.$$ One can show that $p_{\Phi}=p_{L_{\Phi}}$ and $q_{\Phi}=q_{L_{\Phi}}$, see e.g.\ the proof of \cite{Mal85}, Theorem 4.2. For our discussion of $\Phi$-moment inequalities we will need the following results on Orlicz functions. We say that an Orlicz function satisfies the \textit{global
$\Delta_2$-condition} if for some constant $C>0$,
\begin{equation}
\label{eqn:D2global}
\Phi(2t) \leq C \Phi(t) \qquad (t\geq 0).
\end{equation}
Under this condition we have, for any $\al\geq 0$,
$$\Phi(\al t)\lesssim_{\al,\Phi} \Phi(t) \qquad (t\geq 0).$$
One can show (\cite{Mal85}, Theorem 3.2(b)) that (\ref{eqn:D2global}) is equivalent to the assumption $q_{\Phi}<\infty$, which in turn holds if and only if
\begin{equation}
\label{eqn:derivOrlicz} \sup_{t>0} \frac{t\Phi'(t)}{\Phi(t)} < \infty.
\end{equation}
Finally, we shall use the following characterization of Boyd's indices for Orlicz spaces (\cite{Mal85}, Theorem 6.4):}
\begin{equation}
\label{eqn:charOrlInd}
\begin{split}
p_{\Phi} & = \sup\Big\{p>0 \ : \ \int_0^t s^{-p}\Phi(s)\frac{ds}{s} = O(t^{-p}\Phi(t)) \ \forall t>0\Big\}, \\
q_{\Phi} & = \inf\Big\{q>0 \ : \ \int_t^{\infty} s^{-q}\Phi(s)\frac{ds}{s} = O(t^{-q}\Phi(t)) \ \forall t>0\Big\}.
\end{split}
\end{equation}
\end{example}
We refer to \cite{BeS88,KPS82,LiT79} for many more concrete examples of symmetric quasi-Banach function spaces.

\section{Characterization of Marcinkiewicz weak type operators}

In this section we establish a key observation, which essentially reduces the proof of Boyd's theorem and its noncommutative extensions to proving a
certain inequality for distribution functions, which is stated in Lemma~\ref{lem:keyDI} below. This observation moreover leads to a characterization of the subconvex
operators which are simultaneously of weak types $(p,p)$ and $(q,q)$, see Theorem~\ref{thm:WTChar}.\par For $0<p,q\leq\infty$ we define the functions
$\phi_q,\psi_p,\theta_{p,q}:\R_+\rightarrow\R_+$ by
\begin{align*}
\phi_q(t) & = t^{-\frac{1}{q}}\chi_{(0,1)}(t) \qquad (t>0);\\
\psi_p(t) & = t^{-\frac{1}{p}}\chi_{[1,\infty)}(t) \qquad (t>0);\\
\theta_{p,q}(t) & = \psi_p(t) + \phi_q(t) \qquad (t>0).
\end{align*}
Here it is understood that $\phi_{\infty}=\chi_{(0,1)}$. Corresponding to these functions we define three linear operators $\Phi_q,\Psi_p,\Theta_{p,q}:S(\R_+)\rightarrow
\tilde{S}(\R_+\ti\R_+)$ by
$$\Phi_q(f) = f\ot \phi_q, \ \ \Psi_p(f) = f\ot \psi_p, \ \ \Theta_{p,q}(f) = f\ot\theta_{p,q}.$$
The following observation is a
reformulation of \cite{DPP11}, Lemma 4.3.
\begin{lemma}
\label{lem:distEst}
Let $E$ be a symmetric quasi-Banach function space on $\R_+$ and let $0<q<\infty$. If $q_E<q$, then $\Phi_q$ is bounded from $E(\R_+)$ into $E(\R_+\ti\R_+)$. Conversely, if $\Phi_q$ is bounded then $q_E\leq q$.
\end{lemma}
Clearly $\Phi_{\infty}$ is an isometry from $E(\R_+)$ into $E(\R_+\ti\R_+)$ for every symmetric quasi-Banach function space $E$.\par The corresponding result for the
lower Boyd index reads as follows. In the proof and later on, we use $\chi_A$ to denote the indicator of a set $A$.
\begin{lemma}
\label{lem:distEstPsi}
Let $E$ be a symmetric quasi-Banach function space on $\R_+$ and let $0<p<\infty$. If $p<p_E$, then $\Psi_p$ is bounded from $E(\R_+)$ into $E(\R_+\ti\R_+)$. Conversely, if $\Psi_p$ is bounded then $p\leq p_E$.
\end{lemma}
\begin{proof}
Fix $p<p_0<p_E$. It clearly suffices to prove
\begin{equation}
\label{eqn:distEstPsi}
\|f\ot\psi_p\|_{E(\R_+\ti \R_+)}\leq c_{p,E} \|f\|_{E(\R_+)},
\end{equation}
for any $f \in E_+$. Observe that $f\ot\chi_{[2^n,2^{n+1})}$ has the same distribution on $\R_+\ti\R_+$ as $D_{2^{-n}}f$ on $\R_+$. Hence,
\begin{eqnarray*}
\|f\ot\psi_p\|_{E(\R_+\ti\R_+)} & \leq & \Big\|f(s)\sum_{n=0}^{\infty} 2^{-\frac{n}{p}}\chi_{[2^{n},2^{n+1})}(t)\Big\|_{E(\R_+\ti\R_+)} \\
& \leq & C\Big(\sum_{n=0}^{\infty} 2^{-\frac{nr}{p}}\|f(s)\chi_{[2^n,2^{n+1})}(t)\|_{E(\R_+\ti\R_+)}^r\Big)^{\frac{1}{r}} \\
& = & C\Big(\sum_{n=0}^{\infty} 2^{-\frac{nr}{p}}\|D_{2^{-n}}f\|_{E(\R_+)}^r\Big)^{\frac{1}{r}},
\end{eqnarray*}
where $C$ and $0<r\leq 1$ are as in (\ref{eqn:Aoki-Rolewicz}). By (\ref{eqn:BoydIndexLowerDef}), there is some constant $C_{p_0}>0$ such that
$$\|D_u\|\leq C_{p_0}u^{-\frac{1}{p_0}} \qquad (0<u\leq 1).$$
Hence,
\begin{eqnarray*}
\|f\ot\psi_p\|_{E(\R_+\ti\R_+)} & \leq & cC_{p_0} \Big(\sum_{n=0}^{\infty} 2^{-\frac{nr}{p}}2^{\frac{nr}{p_0}}\|f\|_{E(\R_+)}^r\Big)^{\frac{1}{r}} \\
& \lesssim_{p,E} & \|f\|_{E(\R_+)},
\end{eqnarray*}
as $\frac{1}{p_0}-\frac{1}{p}<0$.\par
For the second assertion, notice first that $\mu(D_s(f))=D_s\mu(f)$ for all $0<s<\infty$ and $f \in E$. Therefore, it suffices to show that there is a constant $c>0$ such that for all $0<s\leq 1$ and $f \in E_+$ we have $\|D_s f\|_E \leq c s^{-\frac{1}{p}} \|f\|_E$. If $1\leq a<\infty$, then
\begin{eqnarray*}
\|f\ot\psi_p\|_{E(\R_+\ti\R_+)} & \geq & \|f(s)t^{-\frac{1}{p}}\chi_{[a,2a)}(t)\|_{E(\R_+\ti\R_+)} \\
& \geq & \|f(s)(2a)^{-\frac{1}{p}}\chi_{[a,2a)}(t)\|_{E(\R_+\ti\R_+)} \\
& = & 2^{-\frac{1}{p}} a^{-\frac{1}{p}} \|D_{a^{-1}}f\|_{E(\R_+)},
\end{eqnarray*}
where we use that $f\ot\chi_{[a,2a)}$ has the same distribution on $\R_+\ti\R_+$ as $D_{a^{-1}}f$ on $\R_+$. By (\ref{eqn:distEstPsi}) we arrive at
$$\|D_{a^{-1}}f\|_E \leq 2^{\frac{1}{p}}a^{\frac{1}{p}} \|f\ot\psi_p\|_{E(\R_+\ti\R_+)} \lesssim_{p,E} a^{\frac{1}{p}}\|f\|_E.$$
Since this holds for any $1\leq a<\infty$, we conclude that $p\leq p_E$.
\epf
\end{proof}
As a result of Lemmas~\ref{lem:distEst} and \ref{lem:distEstPsi} we find the following novel expressions for Boyd's indices:
\begin{align*}
p_E & = \sup\Big\{p>0 \ : \ \exists C>0 \ \forall f \in E \ \|\Psi_p(f)\|_{E(\R_+\ti\R_+)}\leq C \|f\|_{E(\R_+)}\Big\} \\
q_E & = \inf\Big\{q>0 \ : \ \exists C>0 \ \forall f \in E \ \|\Phi_q(f)\|_{E(\R_+\ti\R_+)}\leq C \|f\|_{E(\R_+)}\Big\}.
\end{align*}
Moreover, we have the following result.
\begin{corollary}
\label{cor:distEstTheta}
Let $0<p<q<\infty$ and let $E$ be a symmetric quasi-Banach function space on $\R_+$. If $p<p_E\leq q_E<q$, then $\Theta_{p,q}$ is bounded from $E(\R_+)$ into $E(\R_+\ti\R_+)$. Conversely, if $\Theta_{p,q}$ is bounded, then $p\leq p_E\leq q_E\leq q$.
\end{corollary}
The bound for the operator norm of $\Theta_{p,q}$ given in the proof of Corollary~\ref{cor:distEstTheta} can be improved in specific situations. For example, if $0<p<r<q\leq\infty$, then one easily calculates that
$$\|\Theta_{p,q}\|_{L^r\rightarrow L^r} = \Big(\frac{p}{r-p} + \frac{q}{q-r}\Big)^{\frac{1}{r}}.$$
We now compute the distribution function of $\Phi_q(f)$, $\Psi_p(f)$ and $\Theta_{p,q}(f)$. The first was already done in \cite{DPP11}, Lemma 4.4.
\begin{lemma}
\label{lem:dist}
Let $0<q<\infty$. If $f\in S(\R_+)$, then for every $v>0$,
$$d(v;\Phi_q(f)) = \int_{\{f\leq v\}} \Big(\frac{f(s)}{v}\Big)^q ds + d(v;f)$$
and
$$d(v;\Phi_{\infty}(f)) = d(v;f).$$
\end{lemma}
\begin{lemma}
\label{lem:distPsi} Let $0<p<\infty$ and $f \in S(\R_+)$. If $d(v;f)<\infty$, then
$$d(v;\Psi_p(f)) = \int_{\{f>v\}}\Big(\frac{f(s)}{v}\Big)^p ds - d(v;f).$$
\end{lemma}
\begin{proof}
Using a change of variable,
\begin{align*}
& \lambda\Big( (s,t) \in \R_+\ti\R_+ \ : \ f(s)\psi_p(t)>v\Big) \\
& \qquad \qquad = \int_{1}^{\infty}\lambda\Big(s \in \R_+ \ : \ f(s)t^{-\frac{1}{p}}>v\Big) dt \\
& \qquad \qquad = \int_{1}^{\infty}\lambda\Big(s \in \R_+ \ : \ \frac{f(s)}{v}>t^{\frac{1}{p}}\Big) dt \\
& \qquad \qquad = \int_{1}^{\infty}\lambda\Big(s \in \R_+ \ : \ \frac{f(s)}{v}>u\Big) pu^{p-1}du \\
& \qquad \qquad = \Big\|\frac{f}{v}\Big\|_{L^p(\R_+)}^p - \int_{0}^{1}\lambda\Big(s \in \R_+ \ : \ \frac{f(s)}{v}>u\Big) pu^{p-1}du.
\end{align*}
Observe that
\begin{align*}
& \int_{0}^{1}\lambda\Big(s \in \R_+ \ : \ \frac{f(s)}{v}>u\Big) pu^{p-1}du \\
& \qquad \qquad =  \int_{0}^{\infty}\lambda\Big(s \in \R_+ \ : \ \min\Big\{\frac{f(s)}{v},1\Big\}>u\Big) pu^{p-1}du
\\
& \qquad \qquad = \|\min\{v^{-1}f,1\}\|_{L^p(\R_+)}^p = \int_{\{f\leq v\}}\Big(\frac{f(s)}{v}\Big)^p ds + d(v;f),
\end{align*}
which gives the conclusion.
\end{proof}
\begin{corollary}
\label{cor:distTheta} Let $0<p,q<\infty$. If $f \in S(\R_+)$, then for any $v>0$,
\begin{equation}
\label{eqn:distThetaLessInfty} d(v;\Theta_{p,q}(f)) = \int_{\{f>v\}}\Big(\frac{f(s)}{v}\Big)^p ds + \int_{\{f\leq v\}}\Big(\frac{f(s)}{v}\Big)^q ds
\end{equation}
and
\begin{equation}
\label{eqn:distThetaInfty}
d(v;\Theta_{p,\infty}(f)) = \int_{\{f>v\}}\Big(\frac{f(s)}{v}\Big)^p ds.
\end{equation}
\end{corollary}
\begin{proof}
Since $f\ot\phi_q$ and $f\ot\psi_p$ have disjoint supports we have $d(v;\Phi_q(f)) + d(v;\Psi_p(f)) = d(v;\Theta_{p,q}(f))$. Therefore, if $d(v;f)<\infty$, then
(\ref{eqn:distThetaLessInfty}) and (\ref{eqn:distThetaInfty}) follow immediately from Lemmas~\ref{lem:dist} and \ref{lem:distPsi}. On the other hand, for any $v>0$
$$d(v;f)\leq d(v;\Phi_q(f)) \leq d(v;\Theta_{p,q}(f))$$
and
$$d(v;f)\leq \int_{\{f>v\}}\Big(\frac{f(s)}{v}\Big)^p ds.$$
Hence, if $d(v;f)=\infty$, then both sides of (\ref{eqn:distThetaLessInfty}) and (\ref{eqn:distThetaInfty}) are equal to $\infty$.\epf
\end{proof}
\begin{lemma}
\label{lem:keyDI} Let $E$ be a symmetric quasi-Banach function space on $\R_+$. Let $\alpha>0$ and $f \in E_+$. Suppose that either $p<p_E\leq q_E<q<\infty$ or $p<p_E$
and $q=\infty$ and $g \in S(\R_+)$ satisfies
\begin{equation}
\label{eqn:keyDIpq}
d(\alpha v;g) \leq d(v;\Theta_{p,q}f) \qquad (v>0).
\end{equation}
Then $g \in E$ and
$$\|g\|_E \leq \alpha \|\Theta_{p,q}\|_{E\rightarrow E} \ \|f\|_E.$$
\end{lemma}
\begin{proof}
We take right continuous inverses in (\ref{eqn:keyDIpq}) to obtain
$$\mu_t(g) \leq \alpha\mu_t(\Theta_{p,q}f) \qquad (t\geq 0).$$
As $E$ is symmetric, it follows from Corollary~\ref{cor:distEstTheta} that $g \in E$ and moreover,
$$\|g\|_{E} \leq \alpha \|\Theta_{p,q}f\|_{E(\R_+\ti\R_+)} \leq \alpha \|\Theta_{p,q}\|_{E\rightarrow E} \ \|f\|_{E}.$$
\end{proof}
The following result is reminiscent of Calder\'{o}n's characterization of weak type operators (see Theorem~\ref{thm:CalChar} for a noncommutative extension). Recall that if $D$ is a convex set in $S(\R_+)$, then an operator $T:D\rightarrow S(\R_+)$ is called \emph{subconvex} if for any $f,g \in D$ and $t \in [0,1]$ we have
$$T(tf + (1-t)g) \leq tT(f) + (1-t)T(g).$$
\begin{theorem}
\label{thm:WTChar}
Let $0<p\leq q\leq\infty$. A subconvex operator $T:L^{p}(\R_+)_+ + L^{q}(\R_+)_+\rightarrow S(\R_+)$ is simultaneously of Marcinkiewicz weak types
$(p,p)$ and $(q,q)$, i.e.,
\begin{equation}
\label{eqn:WTChar1}
\|Tf\|_{L^{r,\infty}(\R_+)} \leq C_r \|f\|_{L^r(\R_+)} \qquad (f \in L^r(\R_+)_+, \ r=p,q)
\end{equation}
if and only if there is some $\al>0$ such that for all $f \in S(\R_+)$,
\begin{equation}
\label{eqn:WTChar2}
d(\al v;Tf) \leq d(v;\Theta_{p,q}(f)) \qquad (v>0).
\end{equation}
\end{theorem}
\begin{proof}
Suppose that (\ref{eqn:WTChar1}) holds and fix $v>0$. We may assume that $d(v;\Theta_{p,q}(f))<\infty$, for otherwise there is nothing to prove. By
Corollary~\ref{cor:distTheta} it follows that $f\chi_{\{f>v\}} \in L^p(\R_+)$ and $f\chi_{\{f\leq v\}} \in L^q(\R_+)$. If $C_{p,q}=\max\{C_p,C_q\}$, then by subconvexity,
\begin{eqnarray*}
d(2C_{p,q}v;Tf) & \leq & d(2C_{p,q}v;\tfrac{1}{2}T(2f\chi_{\{f\leq v\}}) + \tfrac{1}{2}T(2f\chi_{\{f>v\}})) \\
& \leq & d(2C_{p,q}v;T(2f\chi_{\{f\leq v\}})) + d(2C_{p,q}v;T(2f\chi_{\{f>v\}})).
\end{eqnarray*}
By (\ref{eqn:WTChar1}) and Corollary~\ref{cor:distTheta},
\begin{align*}
& d(2C_{p,q}v;Tf) \\
& \qquad \leq (2C_{p,q}v)^{-q}C_q^q \|2f\chi_{\{f\leq v\}}\|_{L^q(\R_+)}^q + (2C_{p,q}v)^{-p}C_p^p \|2f\chi_{\{f>v\}}\|_{L^p(\R_+)}^p \\
& \qquad \leq d(v;\Theta_{p,q}f).
\end{align*}
Suppose now that (\ref{eqn:WTChar2}) holds. If $q<\infty$, then by Corollary~\ref{cor:distTheta},
$$d(\al v;Tf) \leq d(v;\Theta_{p,q}(f)) = \int_{\{f\leq v\}} v^{-q}f(s)^q ds + \int_{\{f>v\}}v^{-p}f(s)^p ds.$$
Since $p\leq q$ we have
$$(v^{-1}f)^p\chi_{\{f>v\}} \leq (v^{-1}f)^q\chi_{\{f>v\}}, \qquad (v^{-1}f)^q\chi_{\{f\leq v\}} \leq (v^{-1}f)^p\chi_{\{f\leq v\}}$$
and therefore,
$$d(\al v;Tf) \leq v^{-r} \|f\|_{L^r(\R_+)}^r \qquad (r=p,q).$$
On the other hand, if $q=\infty$, then it is clear that
$$d(\al v;Tf) \leq d(v;\Theta_{p,\infty}(f)) = \int_{\{f>v\}}v^{-p}f(s)^p ds \leq v^{-p}\|f\|_{L^p(\R_+)}^p.$$
Moreover, for any $v>0$ we have
$$d(\al v;T(f\chi_{\{f\leq v\}})) = 0.$$
Applying this for $v=\|f\|_{\infty}$ yields
$$Tf \leq \alpha \|f\|_{\infty} \qquad \mathrm{a.e.}$$
This completes the proof.
\end{proof}
The following result shows that inequality (\ref{eqn:keyDIpq}) also implies $\Phi$-moment inequalities.
\begin{lemma}
\label{lem:keyDIOrlicz} Let $\Phi$ be an Orlicz function on $\R_+$ which satisfies the global $\Delta_2$-condition. Let $\alpha>0$ and $f \in (L_{\Phi})_+$. Suppose that
either $p<p_{\Phi}\leq q_{\Phi}<q<\infty$ or $p<p_{\Phi}$ and $q=\infty$ and $g \in S(\R_+)$ satisfies (\ref{eqn:keyDIpq}). Then $g \in L_{\Phi}$ and
\begin{equation}
\label{eqn:PhimomentRV}
\int_0^{\infty} \Phi(|g(t)|) dt \lesssim_{\Phi} \int_0^{\infty} \Phi(f(t)) dt.
\end{equation}
\end{lemma}
\begin{proof}
Suppose that $q_{\Phi}<q<\infty$. Let $\lambda_f$ denote the pull-back measure on $\R_+$ associated with $f$ and $\lambda$. By corollary~\ref{cor:distTheta} we can
rewrite (\ref{eqn:keyDIpq}) as
$$d(\alpha v;g) \leq v^{-q}\int_0^v t^q d\lambda_f(t) + v^{-p}\int_v^{\infty} t^p d\lambda_f(t).$$
Integrating with respect to $\Phi$ and using Fubini's theorem yields
\begin{align*}
& \int_0^{\infty} \Phi(|g(t)|) \ dt \\
& \qquad \lesssim_{\Phi} \int_0^{\infty}v^{-q}\int_0^v t^q d\lambda_f(t)d\Phi(v) + \int_0^{\infty}v^{-p}\int_v^{\infty} t^p d\lambda_f(t)d\Phi(v) \\
& \qquad = \int_0^{\infty}\int_t^{\infty} v^{-q}t^q d\Phi(v)d\lambda_f(t) + \int_0^{\infty}\int_0^{t} v^{-p}t^p d\Phi(v)d\lambda_f(t).
\end{align*}
By (\ref{eqn:derivOrlicz}) and (\ref{eqn:charOrlInd}), we find
$$\int_t^{\infty} v^{-q} d\Phi(v) \lesssim_{\Phi} \int_t^{\infty} v^{-q}\Phi(v)\frac{dv}{v} \lesssim_{\Phi} t^{-q}\Phi(t).$$
Similarly,
$$\int_0^{t} v^{-p} d\Phi(v) \lesssim_{\Phi} t^{-p}\Phi(t).$$
We conclude that
$$\int_0^{\infty} \Phi(|g(t)|) dt \lesssim_{\Phi} \int_0^{\infty} \Phi(t)d\lambda_f(t) = \int_0^{\infty} \Phi(f(t)) dt.$$
The statement for $q=\infty$ is proved analogously.
\end{proof}
\begin{remark}
\label{rem:OrliczNonconvex} From the presented proof it is clear that the result in Lemma~\ref{lem:keyDIOrlicz}, and hence the $\Phi$-moment inequalities discussed
below, remain valid if $\Phi$ is non-convex, provided that it satisfies (\ref{eqn:D2global}) and (\ref{eqn:derivOrlicz}), and $p_{\Phi},q_{\Phi}$ are understood as in
(\ref{eqn:charOrlInd}). It should be noted that in this case $L_{\Phi}$ is in general no longer a quasi-Banach space.
\end{remark}
Boyd's interpolation theorem for Marcinkiewicz weak type operators, as well as a $\Phi$-moment version, are now an immediate consequence of the previous observations.
\begin{theorem}
\label{thm:BoydCom}
Fix $0<p\leq q\leq\infty$. Let $T:L^{p}(\R_+)_+ + L^{q}(\R_+)_+\rightarrow S(\R_+)$ be a subconvex operator of Marcinkiewicz weak types $(p,p)$ and $(q,q)$. If $E$ is a symmetric quasi-Banach function space on $\R_+$, and either $p<p_E\leq q_E<q<\infty$ or $p<p_E$ and $q=\infty$ holds, then
$$\|Tf\|_{E(\R_+)} \leq 2 \|\Theta_{p,q}\|_{E\rightarrow E}\max\{C_p,C_q\} \  \|f\|_{E(\R_+)} \qquad (f \in E(\R_+)_+).$$
On the other hand, if $\Phi$ is an Orlicz function on $\R_+$ satisfying the global $\Delta_2$-condition and either $p<p_{\Phi}\leq q_{\Phi}<q<\infty$ or $p<p_{\Phi}$ and $q=\infty$ holds, then
$$\int_0^{\infty} \Phi(|Tf(t)|) dt \lesssim_{\Phi,C_p,C_q} \int_0^{\infty} \Phi(f(t)) dt \qquad (f \in (L_{\Phi})_+).$$
\end{theorem}
\begin{proof}
As we have seen in the proof of Theorem~\ref{thm:WTChar},
$$d(2\max\{C_p,C_q\}v;Tf)\leq d(v;\Theta_{p,q}f).$$
The assertions now follow from Lemmas~\ref{lem:keyDI} and \ref{lem:keyDIOrlicz}, respectively.
\end{proof}
\begin{remark}
Theorem~\ref{thm:BoydCom} can be extended to a vector-valued interpolation theorem using the following simple trick contained in the proof of \cite{BCP62}, Lemma 1. Suppose that $X,Y$ are Banach spaces and $T:L^{p,1}(\R_+;X) + L^{q,1}(\R_+;X)\rightarrow S(\R_+;Y)$ satisfies
\begin{equation*}
\|Tf\|_{L^{r,\infty}(\R_+;Y)} \leq C_r \|f\|_{L^{r,1}(\R_+;X)} \qquad (f \in L^{r,1}(\R_+;X), \ r=p,q).
\end{equation*}
For a fixed $f$ set $k(f)=(f/\|f\|_X)\chi_{\{f\neq 0\}}$ and define the subconvex operator
$$Sg = \|T(gk(f))\|_Y \qquad (g \in L^{p,1}(\R_+) + L^{q,1}(\R_+)).$$
Since $\|k(f)\|_X=1$, it follows that
$$\|Sg\|_{L^{r,\infty}(\R_+)} \leq C_r \|gk(f)\|_{L^{r,1}(\R_+;X)} = C_r \|g\|_{L^{r,1}(\R_+)} \qquad (g \in L^{r,1}(\R_+), \ r=p,q)$$
and hence by the scalar-valued Boyd interpolation theorem,
$$\|Sg\|_{E} \leq 2\max\{C_p, C_q\}\|\Theta_{p,q}\| \ \|g\|_E \qquad (g \in E).$$
Taking $g=\|f\|_X$ yields
\begin{equation*}
\|Tf\|_{E(\R_+;Y)} \leq 2\max\{C_p, C_q\}\|\Theta_{p,q}\| \ \|f\|_{E(\R_+;X)} \qquad (f \in E(\R_+;X)).
\end{equation*}
\end{remark}

\section{Noncommutative Boyd interpolation theorems}
\label{sec:BoydNC}

In this section we prove a noncommutative version of Boyd's theorem, Theorem~\ref{thm:BoydPos} below. We first recall some terminology and preliminary results for
noncommutative symmetric spaces. Let $\cM$ be a semi-finite von Neumann algebra acting on a complex Hilbert space $H$, which is equipped with a normal, semi-finite,
faithful trace $\tr$. The \emph{distribution function} of a closed, densely defined operator $x$ on $H$, which is affiliated with $\cM$, is given by
\begin{equation*}
\label{eqn:distFunOperator}
d(v;x) = \tr(e^{|x|}(v,\infty)) \qquad (v\geq 0),
\end{equation*}
where $e^{|x|}$ is the spectral measure of $|x|$. The \emph{decreasing rearrangement} of $x$ is defined by
\begin{equation*}
\label{eqn:DRoperator}
\mu_t(x) = \inf\{v>0 \ : \ d(v;x)\leq t\} \qquad (t\geq 0).
\end{equation*}
We say that $x$ is $\tr$\textit{-measurable} if $d(v;x)<\infty$ for some $v>0$. We let $S(\tr)$ be the linear space of all $\tr$-measurable operators, which is a metrizable, complete topological $*$-algebra with respect to the measure topology. We denote by $S_0(\tr)$ the linear subspace of all $x \in S(\tr)$ such that $d(v;x)<\infty$ for all $v>0$. One can introduce a partial order on the linear subspace $S(\tr)_h$ of all self-adjoint operators in $S(\tr)$ by setting, for a self-adjoint operator $x$,
$$x \geq 0 \ \mathrm{if \ and \ only \ if} \ \langle x\xi,\xi\rangle_H\geq 0 \ \mathrm{for \ all} \ \xi \in D(x),$$
where $D(x)$ is the domain of $x$ in $H$. We write $x\leq y$ for $x,y \in S(\tr)_h$ if and only if $y-x\geq 0$. Under this partial ordering $S(\tr)_h$ is a partially
ordered vector space. Let $S(\tr)_+$ denote the positive cone of all $x \in S(\tr)_h$ satisfying $x\geq 0$. It can be shown that $S(\tr)_+$ is closed with respect to the
measure topology (\cite{DDP93}, Proposition 1.4).\par
Throughout our exposition, we will tacitly use many properties of distribution functions and decreasing
rearrangements. For the convenience of the reader we collect these facts in the following two propositions. The first result is essentially contained in the proof of
\cite{Nel74}, Theorem 1.
\begin{proposition}
\label{pro:DistProperties}
If $x,y \in S(\tr)$, then:
\begin{enumerate}
\renewcommand{\labelenumi}{(\alph{enumi})}
\item $d(v;x) = d(v;\mu(x))$ for all $v\geq 0$;
\item $d(v+w;x+y) \leq d(v;x) + d(w;y)$ for all $v,w\geq 0$;
\item if $|x|\leq|y|$ then $d(v;x) \leq d(v;y)$ for all $v\geq 0$.
\end{enumerate}
\end{proposition}
\noindent The following properties of decreasing rearrangements can be found in \cite{FaK86}. If $p$ is a projection in $\cM$, then we let $p^{\perp}:=\id - p$ denote its orthogonal complement.
\begin{proposition}
\label{pro:DRProperties}
If $x,y \in S(\tr)$, then:
\begin{enumerate}
\renewcommand{\labelenumi}{(\alph{enumi})}
\item $\mu_t(\lambda x) = |\lambda| \mu_t(x)$ for all $\lambda \in \C$ and $t\geq 0$;
\item $\mu_{s+t}(x+y) \leq \mu_s(x) + \mu_t(y)$ for all $s,t\geq 0$;
\item if $|x|\leq|y|$ then $\mu_t(x) \leq \mu_t(y)$ for all $t\geq 0$;
\item $\mu_t(uxv) \leq \|u\| \ \mu_t(x) \ \|y\|$ for all $u,v\in \cM$ and $t\geq 0$;
\end{enumerate}
If $e=e^{|x|}(v,\infty)$, then
\begin{enumerate}
\renewcommand{\labelenumi}{(\alph{enumi})}
\addtocounter{enumi}{4}
\item $\mu_t(|x|e) = \mu_t(x)\chi_{[0,\tr(e))}(t)$ for all $t\geq 0$
\item $\mu_t(|x|e^{\perp}) = \mu_{t+\tr(e)}(x)$ for all $t\geq 0$, provided $\tr(e)<\infty$.
\end{enumerate}
Finally, suppose that $\phi:[0,\infty)\rightarrow[0,\infty)$ is an increasing function which is left-continuous on $(0,\infty)$ and satisfies $\phi(0)=0$. If we define
$\phi(\infty):=\lim_{t\rightarrow\infty}\phi(t)$, then
\begin{enumerate}
\addtocounter{enumi}{6}
\renewcommand{\labelenumi}{(\alph{enumi})}
\item $\mu(\phi(|x|)) = \phi(\mu(x))$ on $[0,\infty)$.
\end{enumerate}
\end{proposition}
\noindent For a symmetric (quasi-)Banach function space $E$ on $\R_+$, we define
$$E(\cM,\tr) := \{x \in S(\tr): \ \|\mu(x)\|_{E}<\infty\}.$$
We usually denote $E(\cM,\tr)$ by $E(\cM)$ for brevity. We call $E(\cM)$ the \textit{noncommutative (quasi-)Banach function space} associated with $E$ and $\cM$. In the quasi-Banach case these space were first considered by Xu in \cite{Xu91}. The following fundamental result is proved in \cite{KaS08}, Theorem 8.11 (see also \cite{DDP93,Xu91} for earlier proofs of this result under additional assumptions).
\begin{theorem}
\label{thm:KaS}
If $E$ is a symmetric (quasi-)Banach function space on $\R_+$ which is $p$-convex for some $0<p<\infty$, then $E(\cM)$ defines a $p$-convex (quasi-)Banach space under the (quasi-)norm $\|x\|_{E(\cM)}:=\|\mu(x)\|_{E}$. The space $E(\cM)$ is continuously embedded in $S(\tr)$ with respect to the measure topology.
\end{theorem}
\noindent Using the construction above, we obtain noncommutative versions of many important spaces in analysis, such as $L^p$-spaces, weak $L^p$-spaces, Lorentz spaces and Orlicz spaces. For more details on measurable operators we refer to \cite{DPS11,FaK86,Nel74} and for the theory of noncommutative symmetric spaces to \cite{DDP89,DDP92,DDP93,DPS11,KaS08}.\par
We will now proceed to prove the noncommutative version of Boyd's theorem. We first show that the noncommutative symmetric space $E(\cM)$ is intermediate for the couple $(L^p(\cM),L^q(\cM))$ if $p<p_E\leq q_E<q$, using the following observation.
\begin{lemma}
\label{lem:embed} Let $0<p<q\leq\infty$ and let $E$ be a symmetric quasi-Banach function space $\R_+$ which is $r$-convex for some $0<r<\infty$. If $E(\cM)\subset
L^p(\cM) + L^q(\cM)$, then
$$\|x\|_{L^p(\cM)+L^q(\cM)} \lesssim_{p,q,E} \|x\|_{E(\cM)} \qquad (x \in E(\cM)).$$
\end{lemma}
\begin{proof}
By Theorem~\ref{thm:Aoki-RolewiczOrig}, there exists an equivalent $s$-norm on $E(\cM)$ for some $0<s\leq 1$. Suppose the assertion is not true. Then there exist $x_n \in E(\cM)_+$ such that $\|x_n\|_{E(\cM)}\leq 1$, but $\|x_n\|_{L^p(\cM)+L^q(\cM)} > n^{2/s+1}$ for all $n\geq 1$. By completeness it follows that $\sum_{n\geq 1} n^{-2/s}x_n$ converges in $E(\cM)$ to some $x \in E(\cM)_+$ and since $E(\cM)\subset L^p(\cM) + L^q(\cM)$ we have $x \in (L^p(\cM) + L^q(\cM))+$. But $n^{-2/s}x_n\leq x$ and so $n<n^{-2/s}\|x_n\|_{L^p(\cM) + L^q(\cM)}\leq\|x\|_{L^p(\cM) + L^q(\cM)}$, a contradiction.
\epf
\end{proof}
\begin{lemma}
\label{lem:BoydInclusions} Let $0<p<q\leq\infty$ and let $E$ be a symmetric quasi-Banach function space $\R_+$ which is $r$-convex for some $0<r<\infty$. If $0<p<p_E$ and either $q_E<q<\infty$ or $q=\infty$, then
$$L^p(\cM)\cap L^q(\cM) \subset E(\cM) \subset L^p(\cM) + L^q(\cM),$$
with continuous inclusions.
\end{lemma}
\begin{proof}
If $x \in E(\cM)$, then by Corollary~\ref{cor:distEstTheta} we have $\Theta_{p,q}\mu(x) \in E$ and hence $d(v;\Theta_{p,q}\mu(x))<\infty$ for some $v>0$. If
$e_v=e^{|x|}[0,v]$, then by Proposition~\ref{pro:DRProperties}
$$\|xe_v\|_{L^q(\cM)}^q = \int_{\{\mu(x)\leq v\}} \mu_t(x)^q dt, \qquad \|xe_v^{\perp}\|_{L^p(\cM)}^p = \int_{\{\mu(x)>v\}} \mu_t(x)^p dt.$$
It therefore follows from Corollary~\ref{cor:distTheta} that
\begin{equation*}
v^{-q}\|xe_v\|_{L^q(\cM)}^q + v^{-p}\|xe_v^{\perp}\|_{L^p(\cM)}^p = d(v;\Theta_{p,q}\mu(x))<\infty.
\end{equation*}
Hence $x \in L^p(\cM)+L^{q}(\cM)$. By Lemma~\ref{lem:embed} this implies that $E(\cM)\subset L^p(\cM)+L^{q}(\cM)$ continuously.\par Suppose now that $q=\infty$. Pick
$v>0$ such that $d(v;\Theta_{p,\infty}\mu(x))<\infty$. Then $xe_v \in \cM$ and $xe_v^{\perp} \in L^p(\cM)$ since by Proposition~\ref{pro:DRProperties} and
Corollary~\ref{cor:distTheta},
$$v^{-p}\|xe_v^{\perp}\|_{L^p(\cM)}^p = v^{-p} \int_{\{\mu(x)>v\}} \mu_t(x)^p dt = d(v;\Theta_{p,\infty}\mu(x)).$$
By Lemma~\ref{lem:embed} we conclude that $E(\cM)$ embeds continuously into $L^p(\cM)+\cM$.\par The first inclusion is immediate from the commutative case (see e.g.\ \cite{LiT79}, Proposition 2.b.3), as $(L^p\cap L^q)(\cM) = L^p(\cM) \cap L^q(\cM)$.
\end{proof}
To formulate our main result the following definition is convenient. The notion of subconvexity given below weakens the notion of sublinear operators on spaces of measurable operators introduced by Q.\ Xu (see \cite{Hu09}, where it first appeared in published form).
\begin{definition}
\label{def:subconv}
Let $\cM$ and $\cN$ be von Neumann algebras equipped with normal, semi-finite, faithful traces $\tr$ and $\si$, respectively. Let $D$ be a convex subset of $S(\tr)$. A map $T:D\rightarrow S(\si)_h$ is called \emph{midpoint convex} if
$$T(\tfrac{1}{2}x + \tfrac{1}{2}y) \leq \tfrac{1}{2}T(x) + \tfrac{1}{2}T(y)$$
for all $x,y \in D$. A map $U:D\rightarrow S(\si)$ is called \emph{midpoint subconvex} if for every $x,y \in D$ there exist partial isometries $u,v \in \cN$ such that
$$|U(\tfrac{1}{2}x+\tfrac{1}{2}y)|\leq \tfrac{1}{2}u^*|Ux|u + \tfrac{1}{2}v^*|Uy|v.$$
\end{definition}
\noindent It is a well-known fact (see e.g. \cite{FaK86}, Lemma 4.3) that for any $x,y \in S(\si)$ there are partial isometries $u,v \in \cN$ such that
$$|x+y| \leq u^*|x|u + v^*|y|v.$$
Therefore, any linear map is (midpoint) subconvex.\par
For further reference we state Chebyshev's inequality.
\index{inequality!Chebyshev}
\begin{lemma}
\label{lem:Cheb} (Chebyshev\textquoteright s inequality) Let $0<q<\infty$. If $x \in L^q(\cM)$, then
$$d(v;x) \leq v^{-q}\|x\|_{L^q(\cM)}^q \qquad (v>0).$$
\end{lemma}
\noindent For any $0<r<\infty$,
\begin{equation}
\label{eqn:LrinftyDist}
\|x\|_{L^{r,\infty}(\cM)} = \sup_{t>0} t^{\frac{1}{r}} \mu_t(x) = \sup_{v>0} v \ d(v;x)^{\frac{1}{r}},
\end{equation}
so Chebyshev's inequality implies that $L^r(\cM)\subset L^{r,\infty}(\cM)$ contractively.
\begin{theorem}
\label{thm:BoydPos} Let $E$ be a symmetric quasi-Banach function space on $\R_+$ which is $s$-convex for some $0<s<\infty$. Let $\cM,\cN$ be von Neumann algebras
equipped with normal, semi-finite, faithful traces $\tr$ and $\si$, respectively. Suppose that $0<p<q\leq\infty$ and let $T:L^{p}(\cM)_+ + L^{q}(\cM)_+\rightarrow
S(\si)$ be a midpoint subconvex map such that for some constants $C_p, C_q>0$ depending only on $p$ and $q$, respectively,
\begin{equation}
\label{eqn:BoydPosAss}
\|Tx\|_{L^{r,\infty}(\cN)} \leq C_r \|x\|_{L^r(\cM)} \qquad (x \in L^r(\cM)_+, \ r=p,q).
\end{equation}
If $p<p_E\leq q_E<q<\infty$ or $p<p_E$ and $q=\infty$, then
$$\|Tx\|_{E(\cN)} \leq 2\|\Theta_{p,q}\|\max\{C_p, C_q\} \ \|x\|_{E(\cM)} \qquad (x \in E(\cM)_+).$$
The same result holds if $T:L^{p}(\cM)_+ + L^{q}(\cM)_+\rightarrow S(\si)_h$ is a midpoint convex map satisfying (\ref{eqn:BoydPosAss}).
\end{theorem}
\begin{remark}
Property (\ref{eqn:BoydPosAss}) is clearly the noncommutative version of Marcinkiewicz weak type $(r,r)$. The reader should be warned, however, that in the noncommutative literature it is nowadays customary to simply refer to this property as `weak type $(r,r)$'.
\end{remark}
\begin{proof}
We may assume that $\max\{C_p,C_q\}\leq 1$. By Lemma~\ref{lem:BoydInclusions} $T$ is well-defined on $E(\cM)_+$. Let $x \in E(\cM)_+$ and let $e_v = e^{x}[0,v]$. By midpoint subconvexity, there exist partial isometries $u_1,u_2 \in \cN$ such that
$|Tx| \leq \tfrac{1}{2}u_1^*|T(2xe_v)|u_1 + \tfrac{1}{2}u_2^*|T(2xe_v^{\perp})|u_2.$
It follows that
\begin{eqnarray}
\label{eqn:splitBoyd}
d(2v;Tx) & \leq & d(v;\tfrac{1}{2}u_1^*|T(2xe_v)|u_1) + d(v;\tfrac{1}{2}u_2^*|T(2xe_v^{\perp})|u_2) \nonumber\\
& \leq & d(2v;T(2xe_v)) + d(2v;T(2xe_v^{\perp})).
\end{eqnarray}
Suppose first that $q_E<q<\infty$. By (\ref{eqn:LrinftyDist}) and (\ref{eqn:BoydPosAss}) we have
\begin{equation*}
d(v;Ty) \leq v^{-r} C_r^r\|y\|_{L^r(\cM)}^r \qquad (v>0, \ y \in L^r(\cM)_+, \ r=p,q).
\end{equation*}
Therefore,
$$d(2v;Tx) \leq \max\{C_q^q, C_p^p\} \Big((2v)^{-q} \|2xe_v\|_{L^q(\cM)}^q + (2v)^{-p} \|2xe_v^{\perp}\|_{L^p(\cM)}^p\Big)$$
and from Proposition~\ref{pro:DRProperties} it follows that
$$\|xe_v\|_{L^q(\cM)}^q = \int_{\{\mu(x)\leq v\}} \mu_t(x)^q dt, \qquad \|xe_v^{\perp}\|_{L^p(\cM)}^p = \int_{\{\mu(x)>v\}} \mu_t(x)^p dt.$$
Therefore, by Corollary~\ref{cor:distTheta},
\begin{eqnarray*}
d(2v;Tx) & \leq & v^{-q} \int_{\{\mu(x)\leq v\}} \mu_t(x)^q dt + v^{-p} \int_{\{\mu(x)>v\}} \mu_t(x)^p dt \\
& = & d(v;\Theta_{p,q}\mu(x)).
\end{eqnarray*}
The result now follows from Lemma~\ref{lem:keyDI}, using that $d(v;Tx)=d(v;\mu(Tx))$.\par Suppose now that $q=\infty$. Then
$$\|\tfrac{1}{2}u_1^*T(2xe_v)u_1\|_{L^{\infty}(\cN)} \leq C_{\infty}\|xe_v\|_{L^{\infty}(\cM)} \leq v,$$
so $d(v;\tfrac{1}{2}u_1^*T(xe_v)u_1)=0$. By (\ref{eqn:LrinftyDist}) and (\ref{eqn:BoydPosAss}) we have
\begin{equation*}
d(v;Ty) \leq v^{-p} C^p\|y\|_{L^p(\cM)}^p \qquad (v>0, \ y \in L^p(\cM)_+),
\end{equation*}
and therefore (\ref{eqn:splitBoyd}) implies that
\begin{eqnarray*}
d(2v;Tx) & \leq & C_p^p (2v)^{-p} \|2xe_v^{\perp}\|_{L^p(\cM)}^p \\
& \leq & v^{-p} \int_{\{\mu(x)>v\}} \mu_t(x)^p dt = d(v;\Theta_{p,\infty}\mu(x)).
\end{eqnarray*}
Lemma~\ref{lem:keyDI} gives the conclusion. \epf
\end{proof}
It is clear from the proof of Theorem~\ref{thm:BoydPos} that the same result holds for midpoint (sub)convex operators on $L^p(\cM)_h + L^q(\cM)_h$ or $L^p(\cM)+L^q(\cM)$
instead of $L^p(\cM)_+ + L^q(\cM)_+$.\par The original version of Boyd's theorem allows for the interpolation of operators of weak-type $(p,p)$, i.e., which are bounded
from $L^{p,1}$ into $L^{p,\infty}$. Theorem~\ref{thm:BoydPos} only applies for Marcinkiewicz weak type $(p,p)$ operators. In Theorem~\ref{thm:BoydOrigApp} in the
appendix we will show how to obtain a full noncommutative analogue of Boyd's theorem using a different approach.

\subsection{Interpolation of noncommutative probabilistic inequalities}

To illustrate the flexibility of the method used to prove Theorem~\ref{thm:BoydPos}, we modify it to interpolate several noncommutative probabilistic inequalities. In particular we prove the dual version of Doob's maximal inequality in noncommutative symmetric spaces, see Corollary~\ref{cor:dualDoobSymm} below. The latter result is a consequence of Theorem~\ref{thm:l1Int}, which we will interpret in the following section as an interpolation result for operators on noncommutative $l^1$-valued symmetric spaces. For its proof, we shall need the following observation.
\begin{lemma}
\label{lem:projDom}
Let $x \in S(\tr)_+$. If $e$ is a projection in $\cM$, then
$$x \leq 2(exe + e^{\perp}xe^{\perp}).$$
\end{lemma}
\begin{proof}
By writing
$$x = exe + e^{\perp}xe + exe^{\perp} + e^{\perp}xe^{\perp},$$
we see that the asserted inequality is equivalent to
$$exe - e^{\perp}xe - exe^{\perp} + e^{\perp}xe^{\perp}\geq 0.$$
But $x\geq 0$, so
$$exe - e^{\perp}xe - exe^{\perp} + e^{\perp}xe^{\perp} = (x^{\frac{1}{2}}e - x^{\frac{1}{2}}e^{\perp})^*(x^{\frac{1}{2}}e - x^{\frac{1}{2}}e^{\perp}) \geq 0$$
and the result follows.
\epf
\end{proof}
\begin{theorem}
\label{thm:l1Int} Let $E$ be a symmetric quasi-Banach function space
on $\R_+$ which is $s$-convex for some $0<s<\infty$. Let $\cM,\cN$
be von Neumann algebras equipped with normal, semi-finite, faithful
traces $\tr$ and $\si$, respectively. Suppose that $0<p<q\leq\infty$
and for every $k\geq 1$ let $T_k:L^{p}(\cM)_+ +
L^{q}(\cM)_+\rightarrow S(\si)_+$ be positive midpoint convex maps
such that for some constants $C_p, C_q>0$ depending only on $p$ and
$q$, respectively,
\begin{equation}
\label{eqn:BoydPosAssl1} \Big\|\sum_{k\geq 1} T_k(x_k)\Big\|_{L^{r,\infty}(\cN)} \leq C_r \Big\|\sum_{k\geq 1} x_k\Big\|_{L^r(\cM)} \ \ \ (x_k \in L^r(\cM)_+, k\geq 1, \
r=p,q).
\end{equation}
If $p<p_E\leq q_E<q<\infty$ or $p<p_E$ and $q=\infty$, then for any sequence $(x_k)_{k\geq 1}$ in $E(\cM)_+$,
\begin{equation}
\label{eqn:l1Int} \Big\|\sum_{k\geq 1} T_k(x_k)\Big\|_{E(\cN)} \leq 4\|\Theta_{p,q}\| \ \max\{C_p, C_q\} \ \Big\|\sum_{k\geq 1} x_k\Big\|_{E(\cM)},
\end{equation}
where the sums converge in norm.
\end{theorem}
\begin{proof}
We may assume $C_p,C_q\leq 1$. Suppose first that $q_E<q<\infty$. By completeness it suffices to prove (\ref{eqn:l1Int}) for a finite sequence $(x_k)$ in $E(\cM)_+$. Set
$x=\sum_k x_k$. For any $v\geq 0$, let $e_v = e^{x}[0,v]$. By Lemma~\ref{lem:projDom} and positivity and convexity of the $T_k$,
\begin{eqnarray*}
\sum_k T_k(x_k) & \leq & \sum_k T_k(2e_vx_ke_v +
2e_v^{\perp}x_ke_v^{\perp}) \\
& \leq & \tfrac{1}{2} \sum_k T_k(4e_vx_ke_v) + \tfrac{1}{2} \sum_k T_k(4e_v^{\perp}x_ke_v^{\perp}).
\end{eqnarray*}
Therefore,
$$d\Big(4v;\sum_k T_k(x_k)\Big) \leq d\Big(4v;\sum_k T_k(4e_vx_ke_v)\Big) + d\Big(4v;\sum_k T_k(4e_v^{\perp}x_ke_v^{\perp})\Big).$$
By (\ref{eqn:BoydPosAssl1}),
\begin{eqnarray*}
d\Big(4v;\sum_k T_k(x_k)\Big)
& \leq & (4v)^{-q} \Big\|\sum_k 4e_vx_ke_v\Big\|_{L^q(\cM)}^q + (4v)^{-p}\Big\|\sum_k 4e_v^{\perp}x_ke_v^{\perp}\Big\|_{L^p(\cM)}^p \\
& = & v^{-q} \int_{\{\mu(x)\leq v\}} \mu_t(x)^q dt + v^{-p} \int_{\{\mu(x)>v\}} \mu_t(x)^p dt \\
& = & d(v;\Theta_{p,q}\mu(x)),
\end{eqnarray*}
where the final equality follows from Corollary~\ref{cor:distTheta}. The result is now immediate from Lemma~\ref{lem:keyDI}. The case $q=\infty$ follows analogously as
in the proof of Theorem~\ref{thm:BoydPos}. \epf
\end{proof}
As an application of Theorem~\ref{thm:l1Int}, we can interpolate the following dual Doob inequality in noncommutative $L^p$-spaces, due to M.\ Junge.
\begin{theorem}
\label{thm:dualDoobLp} \cite{Jun02} Let $\cM$ be a semi-finite von Neumann algebra and let $(\cE_k)_{k\geq 1}$ be an increasing sequence of conditional expectations in
$\cM$. If $1\leq p<\infty$, then for any sequence $(x_k)_{k\geq 1}$ in $L^p(\cM)_+$,
$$\Big\|\sum_k \cE_k(x_k)\Big\|_{L^p(\cM)} \lesssim_p \Big\|\sum_k x_k\Big\|_{L^p(\cM)}.$$
\end{theorem}
Theorems~\ref{thm:dualDoobLp} and \ref{thm:l1Int} together yield the following extension.
\begin{corollary}
\label{cor:dualDoobSymm} Let $E$ be a symmetric quasi-Banach
function space on $\R_+$ which is $s$-convex for some $0<s<\infty$
and let $\cM$ be a semi-finite von Neumann algebra. Let
$(\cE_k)_{k\geq 1}$ be an increasing sequence of conditional
expectations in $\cM$. If $1<p_E\leq q_E<\infty$, then for any
sequence $(x_k)_{k\geq 1}$ in $E(\cM)_+$,
\begin{equation}
\label{eqn:dualDoobSymm} \Big\|\sum_{k\geq 1}
\cE_k(x_k)\Big\|_{E(\cM)} \lesssim_E \Big\|\sum_{k\geq 1}
x_k\Big\|_{E(\cM)},
\end{equation}
where the sums converge in norm.
\end{corollary}
In \cite{JuX03}, Theorem 7.1, it was shown that any conditional expectation $\cE$ is `anti-bounded' for the $L^p$-norm if $0<p<1$, i.e.,
\begin{equation}
\label{eqn:CEantiBdd}
\|x\|_{L^p(\cM)} \leq 2^{\frac{1}{p}} \|\cE(x)\|_{L^p(\cM)} \qquad (x \in \cM).
\end{equation}
Even though (\ref{eqn:CEantiBdd}) does not correspond to the boundedness of an operator, we can still `interpolate' this estimate.
\begin{proposition}
\label{pro:CEantiBdd}
Let $\cM$ be a finite von Neumann algebra and let $\cE$ be a conditional expectation on $\cM$. If $E$ is a symmetric quasi-Banach function space on
$\R_+$ with $q_E<1$, then
$$\|x\|_{E(\cM)} \lesssim_E \|\cE(x)\|_{E(\cM)} \qquad (x \in \cM).$$
\end{proposition}
\begin{proof}
Let $y=\cE(x)$ and for $v>0$ set $e_v = e^{y}[0,v]$. As was remarked after Lemma~\ref{lem:dilation}, we have $p_E>0$ and hence we can fix $0<p<p_E$ and $q_E<q<1$. By Chebyshev's inequality and (\ref{eqn:CEantiBdd}),
\begin{eqnarray*}
d(2^{1+\frac{1}{p}}v;x) & \leq & v^{-q}\|2^{-\frac{1}{p}}xe_v\|_{L^q(\cM)}^q + v^{-p}\|2^{-\frac{1}{p}}xe_v^{\perp}\|_{L^p(\cM)}^p \\
& \leq & v^{-q} \|\cE(xe_v)\|_{L^q(\cM)}^q + v^{-p}\|\cE(xe_v^{\perp})\|_{L^p(\cM)}^p \\
& = & v^{-q} \|\cE(x)e_v\|_{L^q(\cM)}^q + v^{-p}\|\cE(x)e_v^{\perp}\|_{L^p(\cM)}^p \\
& = & v^{-q} \int_{\{\mu(y)\leq v\}} \mu_t(y)^q dt + v^{-p} \int_{\{\mu(y)>v\}} \mu_t(y)^p dt \\
& = & d(v;\Theta_{p,q}\mu(y)),
\end{eqnarray*}
where the final equality follows from Corollary~\ref{cor:distTheta}. The conclusion now follows from Lemma~\ref{lem:keyDI}.
\end{proof}
The following result facilitates interpolation of noncommutative square function estimates.
\begin{theorem}
\label{thm:l2Int}
Let $E$ be a symmetric quasi-Banach function space on $\R_+$ which is $s$-convex for some $0<s<\infty$. Let $\cM,\cN$ be von Neumann algebras equipped with normal, semi-finite, faithful traces $\tr$ and $\si$, respectively. Suppose that $0<p<q\leq\infty$ and for $k\geq 1$ let $T_k:L^{p}(\cM) + L^{q}(\cM)\rightarrow S(\si)$ be linear maps such that for some constants $C_p, C_q>0$ depending only on $p$ and $q$, respectively,
\begin{equation*}
\label{eqn:BoydPosAssl2} \Big\|\sum_{k\geq 1} T_k(x_k)\Big\|_{L^{r,\infty}(\cN)} \leq C_r \Big\|\Big(\sum_{k\geq 1}|x_k|^2\Big)^{\frac{1}{2}}\Big\|_{L^r(\cM)} \ \ \ (x_k
\in L^r(\cM)_+, k\geq 1, r=p,q).
\end{equation*}
If $p<p_E\leq q_E<q<\infty$ or $p<p_E$ and $q=\infty$, then for any finite sequence $(x_k)$ in $E(\cM)$
$$\Big\|\sum_{k\geq 1} T_k(x_k)\Big\|_{E(\cN)} \leq 2\|\Theta_{p,q}\|\max\{C_p, C_q\} \ \Big\|\Big(\sum_{k\geq 1}|x_k|^2\Big)^{\frac{1}{2}}\Big\|_{E(\cM)}.$$
\end{theorem}
\begin{proof}
We may assume $C_p,C_q\leq 1$. Let $x=(\sum_{k\geq 1}|x_k|^2)^{\frac{1}{2}}$ and for $v>0$ define $e_v = e^x[0,v]$. Then,
$$d(2v;\sum_{k\geq 1}T_k(x_k)) \leq  d(v;\sum_{k\geq 1}T_k(x_ke_v^{\perp})) + d(v;\sum_{k\geq 1}T_k(x_ke_v)).$$
By (\ref{eqn:BoydPosAssl2}),
\begin{eqnarray*}
d(2v;\sum_{k\geq 1}T_k(x_k)) & \leq & v^{-p}\Big\|\Big(\sum_{k\geq 1}|x_ke_v^{\perp}|^2\Big)^{\frac{1}{2}}\Big\|_{L^p(\cM)}^p + v^{-q}\Big\|\Big(\sum_{k\geq
1}|x_ke_v|^2\Big)^{\frac{1}{2}}\Big\|_{L^q(\cM)}^q \\
& = & v^{-p}\Big\|\Big(\sum_{k\geq 1}|x_k|^2\Big)^{\frac{1}{2}}e_v^{\perp}\Big\|_{L^p(\cM)}^p + v^{-q}\Big\|\Big(\sum_{k\geq
1}|x_k|^2\Big)^{\frac{1}{2}}e_v\Big\|_{L^q(\cM)}^q \\
& = & v^{-p} \int_{\{\mu(x)>v\}} \mu_t(x)^p dt + v^{-q} \int_{\{\mu(x)\leq v\}} \mu_t(x)^q dt \\
& = & d(v;\Theta_{p,q}\mu(x)).
\end{eqnarray*}
The result now follows from Lemma~\ref{lem:keyDI}. The case $q=\infty$ is similar.
\end{proof}
As a corollary, we find the following version of Stein's inequality for noncommutative symmetric spaces, which will be needed in the proof of Theorem~\ref{thm:BurkRos} below. A different proof of this result was found in \cite{Jia12}.
\begin{corollary}
\label{cor:SteinNC} Let $E$ be a symmetric quasi-Banach function space on $\R_+$ which is $s$-convex for some $0<s<\infty$ and let $\cM$ be a semi-finite von Neumann
algebra. Let $(\cE_k)_{k\geq 1}$ be an increasing sequence of conditional expectations in $\cM$. If $1<p_E\leq q_E<\infty$, then for any finite sequence $(x_k)$ in
$E(\cM)$,
\begin{equation}
\label{eqn:SteinNC}
\Big\|\Big(\sum_{k\geq 1} |\cE_k(x_k)|^2\Big)^{\frac{1}{2}}\Big\|_{E(\cM)} \lesssim_E \Big\|\Big(\sum_{k\geq 1}|x_k|^2\Big)^{\frac{1}{2}}\Big\|_{E(\cM)}.
\end{equation}
\end{corollary}
\begin{proof}
Let $e_{kl}$ be the standard matrix units, let $y_k = x_k \ot e_{k1}$ and let $T_k = \cE_k \ot \id_{B(l^2)}$. Then (\ref{eqn:SteinNC}) is equivalent to
$$\Big\|\sum_{k\geq 1} T_k(y_k)\Big\|_{E(\cM\ot B(l^2))} \lesssim_E \Big\|\Big(\sum_{k\geq 1}|y_k|^2\Big)^{\frac{1}{2}}\Big\|_{E(\cM\ot B(l^2))}.$$
By \cite{PiX97}, Theorem 2.3, this inequality holds if $E=L^p$ with $1<p<\infty$. Hence, the result follows immediately from Theorem~\ref{thm:l2Int}.
\end{proof}
Further examples of probabilistic inequalities which can be interpolated using the presented method are given by the `upper' noncommutative Khintchine inequalities, see
\cite{DPP11}, Theorem 4.1, and \cite{DiR11}, Corollary 2.2.
\begin{remark}
The results in Theorems~\ref{thm:BoydPos}, \ref{thm:l1Int} and \ref{thm:l2Int}, Corollaries~\ref{cor:dualDoobSymm} and \ref{cor:SteinNC}, and
Proposition~\ref{pro:CEantiBdd} all have an appropriate `$\Phi$-moment version'. Indeed, these versions follow immediately by using Lemma~\ref{lem:keyDIOrlicz} instead
of Lemma~\ref{lem:keyDI} in the proofs of the latter results. In particular, by following the proof of Theorem~\ref{thm:BoydPos} and taking
Remark~\ref{rem:OrliczNonconvex} into account, we find an extension of \cite{BeC10}, Theorem 2.1 for non-convex Orlicz functions.
\end{remark}

\section{Interpolation of noncommutative maximal inequalities}

In this section we present a Boyd-type interpolation theorem for noncommutative maximal inequalities. To formulate maximal inequalities in noncommutative symmetric
spaces and their dual versions, we first introduce the two `noncommutative vector-valued symmetric spaces' $E(\cM;l^{\infty})$ and $E(\cM;l^1)$. We define these in
analogy with the noncommutative vector-valued $L^p$-spaces $L^p(\cM;l^{\infty})$ and $L^p(\cM;l^1)$, which were introduced in \cite{Pis98} for hyperfinite von Neumann
algebras, and considered in general in \cite{Jun02}. From now on, we let $E$ be a symmetric Banach function space on $\R_+$.\par We define $E(\cM;l^{\infty})$ to be the
space of all sequences $x=(x_k)_{k\geq 1}$ in $E(\cM)$ for which there exist $a,b \in E^{(2)}(\cM)$ and a bounded sequence $y=(y_k)_{k\geq 1}$ such that
$$x_k = ay_kb \qquad (k\geq 1).$$
For $x \in E(\cM;l^{\infty})$ we define
\begin{equation}
\label{eqn:linftyNorm} \|x\|_{E(\cM;l^{\infty})} = \inf\{\|a\|_{E^{(2)}(\cM)}\sup_{k\geq 1} \|y_k\|_{\infty}\|b\|_{E^{(2)}(\cM)}\},
\end{equation}
where the infimum is taken over all possible factorizations of $x$ as above. We can think of the quantity (\ref{eqn:linftyNorm}) as `$\|\sup_{k\geq 1} x_k\|_{E(\cM)}$',
even though $\sup_{k\geq 1} x_k$ need not be defined at all.\par We define $E(\cM;l^1)$ to be the space of all sequences $x=(x_k)_{k\geq 1}$ in $E(\cM)$ which can be
decomposed as
$$x_k = \sum_{j\geq 1} u_{jk}^*v_{jk} \qquad (k\geq 1)$$
for two families $(u_{jk})_{j,k\geq 1}$ and $(v_{jk})_{j,k\geq 1}$
in $E^{(2)}(\cM)$ satisfying
$$\sum_{j,k} u_{jk}^*u_{jk} \in E(\cM) \ \ \mathrm{and} \ \ \sum_{j,k} v_{jk}^*v_{jk} \in E(\cM),$$
where the series converge in norm. For $x \in E(\cM;l^1)$ we define
$$\|x\|_{E(\cM;l^1)} = \inf\Big\{\Big\|\sum_{j,k} u_{jk}^*u_{jk}\Big\|_{E(\cM)}^{\frac{1}{2}}\Big\|\sum_{j,k} v_{jk}^*v_{jk}\Big\|_{E(\cM)}^{\frac{1}{2}}\Big\},$$
where the infimum is taken over all decompositions of $x$ as above.
In what follows, we will mostly consider elements $x=(x_k)_{k\geq 1}
\in E(\cM;l^1)$ for which $x_k\geq 0$ for all $k$. In this case,
$$\|x\|_{E(\cM;l^1)} = \Big\|\sum_{k\geq 1}x_k\Big\|_{E(\cM)}.$$
The theory for the spaces $E(\cM;l^{\infty})$ and $E(\cM;l^1)$ can be developed in full analogy with the special case $E=L^p$ considered in \cite{Jun02,JuX07, Xu07}. In
fact, most of the basic results follow \verb"verbatim" as soon as we replace $L^p$ by $E$, $L^{p'}$ by $E^{\ti}$, where $\frac{1}{p} + \frac{1}{p'} = 1$, and $L^{2p}$ by
$E^{(2)}$ in the proofs of these results. For example, the following observation is immediate.
\begin{theorem}
\label{thm:Inftyl1Banach} If $E$ is a symmetric Banach function space on $\R_+$, then $E(\cM;l^{\infty})$ and $E(\cM;l^1)$ are Banach spaces.
\end{theorem}
Our strategy to prove a Boyd-type interpolation theorem for maximal inequalities is to dualize Theorem~\ref{thm:l1Int}, which can be viewed as a Boyd-type interpolation
theorem for $l^1$-valued noncommutative symmetric spaces. We shall need the duality stated in Theorem~\ref{thm:Inftyl1Dual} below. The proof is essentially an adaptation
of the Hahn-Banach separation argument in \cite{Jun02}, Proposition 3.6 (see also \cite{Xu07}, Theorem 4.11) to our context. We need the following observation, proved in
\cite{DDP93}, Theorem 5.6 and p.\ 745.
\begin{theorem}
\label{thm:dualSep} If $E$ is a separable symmetric Banach function space on $\R_+$, then $E(\cM)^*=E^{\times}(\cM)$ isometrically, with associated duality bracket given by
$$\langle x,y\rangle = \tr(xy) \qquad (x \in E(\cM), \ y \in E^{\ti}(\cM)).$$
\end{theorem}
Below we will implicitly use the trace property a number of times, i.e., we will use that if $x,y \in S(\tr)$ are such that $xy,yx \in L^1(\cM)$, then $\tr(xy)=\tr(yx)$.
In particular this holds if $x \in E(\cM)$ and $y \in E^{\ti}(\cM)$.
\begin{theorem}
\label{thm:Inftyl1Dual} Let $\cM$ be a semi-finite von Neumann algebra and let $E$ be a separable symmetric Banach function space on $\R_+$. If $y=(y_k) \in
E^{\ti}(\cM;l^{\infty})$ satisfies $y_k\geq 0$ for all $k$, then
\begin{equation}
\label{eqn:maxNormDual} \|y\|_{E^{\ti}(\cM;l^{\infty})} = \sup\Big\{\sum_{k\geq 1} \tr(x_ky_k) \ : \ x_k\in E(\cM)_+, \ \Big\|\sum_{k\geq 1}x_k\Big\|_{E(\cM)}\leq
1\Big\}.
\end{equation}
\end{theorem}
\begin{proof}
We let $S$ denote the supremum on the right hand side of (\ref{eqn:maxNormDual}). Let $y_k = az_kb$ with $a,b\in (E^{\ti})^{(2)}(\cM)$ and $z_k \in \cM$ with
$\|z_k\|_{\infty}\leq 1$ and let $(x_k)$ be a sequence in $E(\cM)_+$. By H\"{o}lder's inequality,
\begin{eqnarray*}
\sum_{k} \tr(x_ky_k) & = & \sum_{k}\tr(x_kaz_kb) = \sum_{k}\tr(bx_k^{\frac{1}{2}}x_k^{\frac{1}{2}}az_k) \\
& \leq & \sum_{k}\|bx_k^{\frac{1}{2}}\|_{L^2(\cM)}\|x_k^{\frac{1}{2}}az_k\|_{L^2(\cM)} \\
& \leq & \Big(\sum_{k}\tr(bx_kb^*)\Big)^{\frac{1}{2}}\Big(\sum_{k}\tr(a^*x_ka)\Big)^{\frac{1}{2}} \\
& \leq & \|b\|_{(E^{\ti})^{(2)}(\cM)}\Big\|\sum_{k}x_k\Big\|_{E(\cM)} \|a\|_{(E^{\ti})^{(2)}(\cM)}.
\end{eqnarray*}
We conclude that $S\leq \|y\|_{E^{\ti}(\cM;l^{\infty})}$.\par Suppose now that $S=1$, we will show that $\|y\|_{E^{\ti}(\cM;l^{\infty})}\leq 1$. Under this assumption, we have for any finite sequence $x=(x_{k})$ in $E(\cM)_+$,
\begin{equation}
\label{eqn:HBsepIneqPos} \sum_{k} \tr(x_ky_k) \leq \Big\|\sum_{k}x_k\Big\|_{E(\cM)}.
\end{equation}
Let $K=\{s \in E^{\ti}(\cM)_+ \ : \ \|s\|_{E^{\ti}(\cM)}\leq 1\}$, equipped with the weak$^*$ topology of $E(\cM)^*$. Since $E^{\ti}(\cM)$ is isometrically isomorphic to
$E(\cM)^*$ by Theorem~\ref{thm:dualSep} and $E^{\ti}(\cM)_+$ is weak$^*$ closed in $E^{\ti}(\cM)$, we conclude that $K$ is compact by the Banach-Alaoglu theorem. Moreover,
$$\|w\|_{E(\cM)} = \sup_{s \in K} \tr(ws) \qquad (w \in E(\cM)_+).$$
For any finite sequence $x$ as above we define
$$f_{x}(s) = \sum_{k}\tr(x_ks) - \sum_{k}\tr(x_ky_k).$$
Clearly $f_{x}$ is a real-valued continuous function on $K$ and from (\ref{eqn:HBsepIneqPos}) it follows that $\sup_{s\in K} f_x(s)\geq 0$. Let $A$ be the subset of
$C(K)$ consisting of all $f_{x}$, where $x=(x_{k})$ is a finite sequence in $E(\cM)_+$. Then $A$ is a cone in $C(K)$. Indeed, if $\lambda\geq 0$ then $\lambda f_{x} =
f_{\lambda x}$. Moreover, if $x,\tilde{x}$ are finite families in $E(\cM)$, then $f_{x} + f_{\tilde{x}}$ can be realized as $f_{x+\tilde{x}}$, since without loss of
generality we may assume that $\tilde{x}$ is to the right of the finite family $x$. Observe that $A$ is disjoint from the cone $A_{-} = \{g \in C(K) \ : \ \sup g<0\}$.
By the Hahn-Banach separation theorem, there exists a real Borel measure $\mu$ on $K$ and $\al \in \R$ such that for all $f \in A$ and $g \in A_{-}$,
$$\int_{K}g d\mu \leq \al \leq \int_{K} f \ d\mu.$$
Note that $\al=0$, as both $A$ and $A_{-}$ are cones. If $B$ is a Borel subset, then we can find a sequence $(g_i)$ in $A_{-}$ such that $g_n\uparrow -\chi_B$. This
shows that $\mu$ must be positive, and by normalization we may assume that $\mu$ is a probability measure. Hence, for all $k\geq 1$, we have
\begin{equation}
\label{eqn:intRep} \tr(xy_k) \leq \int_{K} \tr(xs) d\mu(s) \qquad (x \in E(\cM)_+).
\end{equation}
Define a positive operator by
$$a = \int_{K} s d\mu(s).$$
Clearly $a \in K$, so $a \in E^{\ti}(\cM)_+$ and $\|a\|_{E^{\ti}(\cM)}\leq 1$. By (\ref{eqn:intRep}) and normality of $\tr$,
\begin{equation*}
\tr(xy_k) \leq \tr(xa) \qquad (x \in E(\cM)_+).
\end{equation*}
This implies that $y_k\leq a$ and therefore we find a contraction $u_k \in \cM$ such that $y_k^{\frac{1}{2}}=u_ka^{\frac{1}{2}}$. In particular, $y_k =
a^{\frac{1}{2}}u_k^*u_ka^{\frac{1}{2}}$ and hence
$$\|y\|_{E^{\ti}(\cM;l^{\infty})} \leq \|a\|_{E^{\ti}(\cM)} \leq 1.$$
This completes the proof.
\end{proof}
\begin{remark}
Using a slightly more involved version of the separation argument in the proof of Theorem~\ref{thm:Inftyl1Dual} (see the proof of \cite{Jun02}, Proposition 3.6, for the
case $E=L^p$), one may show that in fact
$$E(\cM;l^1)^* = E^{\ti}(\cM;l^{\infty})$$
isometrically, with respect to the duality bracket
$$\langle x,y\rangle = \sum_{k\geq 1} \tr(x_ky_k),$$
where $x \in E(\cM;l^1)$ and $y \in E^{\ti}(\cM;l^{\infty})$.
\end{remark}
The following result facilitates the interpolation of noncommutative maximal inequalities.
\begin{theorem}
\label{thm:linftyInt} Let $E$ be the K\"{o}the dual of a separable symmetric Banach function space on $\R_+$. Suppose that $1\leq p<q<\infty$ and let $S_k:L^{p,1}(\cM)_+
+ L^{q,1}(\cM)_+\rightarrow S(\tr)_+$ be positive linear operators satisfying
\begin{equation}
\label{eqn:maxBdd} \|(S_k(x))_{k\geq 1}\|_{L^r(\cM;l^{\infty})} \lesssim_r \|x\|_{L^{r,1}(\cM)} \qquad (x \in L^{r,1}(\cM)_+, r=p,q).
\end{equation}
If $q_E<q$ and either $p=1$ or $p_E>p$, then
\begin{equation}
\label{eqn:linftyIntConc}
\|(S_k(x))_{k\geq 1}\|_{E(\cM;l^{\infty})} \lesssim_E \|x\|_{E(\cM)} \qquad (x \in E(\cM)_+).
\end{equation}
\end{theorem}
\begin{proof}
Let $F$ be the symmetric space on $\R_+$ such that $F^{\ti} = E$. By (\ref{eqn:dualityBoydIndices}) we have $p_F>q'$ and, if $p_E>1$ also $q_F<p'$. Since $S_k$ is
positive, so is its adjoint $S_k^*$. If $r\in\{p,q\}$ and $y=(y_k) \in L^{r'}(\cM;l^1)$ with $y_k\geq 0$, then for any $x\in L^{r,1}(\cM)_+$,
\begin{eqnarray*}
\sum_{k\geq 1} \tr(S_k^*(y_k)x) & = & \sum_{k\geq 1} \tr(y_kS_k(x)) \\
& \leq & \|y\|_{L^{r'}(\cM;l^1)} \|(S_k(x))_{k\geq 1}\|_{L^r(\cM;l^{\infty})} \\
& \lesssim_r & \|y\|_{L^{r'}(\cM;l^1)} \|x\|_{L^{r,1}(\cM)}.
\end{eqnarray*}
It follows that
$$\Big\|\sum_{k\geq 1} S_k^*(y_k)\Big\|_{L^{r',\infty}(\cM)} \lesssim_p \Big\|\sum_{k\geq 1} y_k\Big\|_{L^{r'}(\cM)}.$$
Therefore, if $(y_k) \in F(\cM;l^1)$ satisfies $y_k\geq 0$, then by
Theorem~\ref{thm:l1Int},
$$\Big\|\sum_{k\geq 1} S_k^*(y_k)\Big\|_{F(\cM)} \lesssim_F \Big\|\sum_{k\geq 1} y_k\Big\|_{F(\cM)}.$$
Hence, if $x \in E(\cM)_+$, then
\begin{eqnarray*}
\sum_{k\geq 1} \tr(y_kS_k(x)) & = & \sum_{k\geq 1} \tr(S_k^*(y_k)x) \\
& \leq & \Big\|\sum_{k\geq 1} S_k^*(y_k)\Big\|_{F(\cM)} \|x\|_{E(\cM)} \lesssim_E \Big\|\sum_{k\geq 1} y_k\Big\|_{F(\cM)} \|x\|_{E(\cM)}.
\end{eqnarray*}
By Theorem~\ref{thm:Inftyl1Dual} we conclude that (\ref{eqn:linftyIntConc}) holds.
\end{proof}
Examples of sequences of operators satisfying the conditions of Theorem~\ref{thm:linftyInt} are established in \cite{JuX07}. We give two examples which yield maximal
ergodic inequalities in noncommutative symmetric spaces. Let $T:\cM\rightarrow \cM$ be a linear map such that
\begin{enumerate}
\item $T$ is a contraction on $\cM$;
\item $T$ is positive;
\item $\tr(T(x))\leq \tr(x)$ for all $x \in L^1(\cM)\cap\cM_+$.
\end{enumerate}
In \cite{JuX07}, Theorem 4.1, it is shown that for any $1<p\leq\infty$ the ergodic averages
$$M_k(T) = \frac{1}{k+1}\sum_{i=0}^k T^i \qquad (k\geq 1),$$
satisfy the maximal inequality
$$\|(M_k(T)(x))_{k\geq 1}\|_{L^p(\cM;l^{\infty})} \lesssim_p \|x\|_{L^p(\cM)} \qquad (x \in L^p(\cM)_+).$$
If $T$ moreover satisfies
\begin{enumerate}
\renewcommand{\labelenumi}{(\alph{enumi})}
\addtocounter{enumi}{3}
\item $\tr(T(y)^*x) = \tr(y^*T(x))$ for all $x,y \in L^2(\cM)\cap\cM$,
\end{enumerate}
then, as was observed in \cite{JuX07}, Theorem 5.1, for every
$1<p\leq\infty$ one has
$$\|(T^k(x))_{k\geq 1}\|_{L^p(\cM;l^{\infty})} \lesssim_p \|x\|_{L^p(\cM)} \qquad (x \in L^p(\cM)_+).$$
Using Theorem~\ref{thm:linftyInt} we can interpolate these inequalities to obtain the following result.
\begin{theorem}
\label{thm:maxErgod}
Let $E$ be the K\"{o}the dual of a separable symmetric Banach function space on $\R_+$ and suppose that $1<p_E\leq q_E<\infty$. If $T:\cM\rightarrow\cM$ is a linear
operator satisfying conditions (a)-(c) above, then
$$\|(M_k(T)(x))_{k\geq 1}\|_{E(\cM;l^{\infty})} \lesssim_E \|x\|_{E(\cM)} \qquad (x \in E(\cM)_+).$$
If $T$ moreover satisfies condition (d), then
$$\|(T^k(x))_{k\geq 1}\|_{E(\cM;l^{\infty})} \lesssim_E \|x\|_{E(\cM)} \qquad (x \in E(\cM)_+).$$
\end{theorem}
To conclude this section, we prove a version of Doob's maximal inequality for noncommutative symmetric spaces. First recall the following definitions. Let $E$ be a symmetric Banach function space on $\R_+$ and let $\cM$ be a semi-finite von Neumann algebra. Suppose that $(\cM_k)_{k\geq 1}$ is a filtration, i.e., an increasing sequence of von Neumann subalgebras such that $\tr|_{\cM_k}$ is semi-finite, and let $\cE_k$ be the conditional
expectation with respect to $\cM_k$. A sequence $(y_k)$ in $E(\cM)$ is called a \emph{martingale} with respect to $(\cM_k)$ if
$\cE_{k}(y_{k+1})=y_k$ for all $k\geq 1$. We say that $(y_k)$ is \emph{finite} if there is an $n\geq 1$ such that $y_k=y_n$ for all $k\geq n$. A sequence $(x_k)$ in $E(\cM)$ is called a \emph{martingale difference sequence} if $x_k = y_k - y_{k-1}$ for some martingale $(y_k)$, with the convention $y_{0}=0$ and $\cM_0 = \C\id$.\par
It was shown by M.\ Junge in \cite{Jun02} that for every $1<p\leq\infty$,
\begin{equation}
\label{eqn:DoobMaxLp} \|(\cE_k(y))_{k\geq 1}\|_{L^p(\cM;l^{\infty})}
\lesssim_p \|y\|_{L^p(\cM)}.
\end{equation}
This result implies the following version for noncommutative symmetric spaces.
\begin{theorem}
\label{thm:DoobMax} Let $\cM$ be a semi-finite von Neumann algebra and let $E$ be the K\"{o}the dual of a separable symmetric Banach function space on $\R_+$ with
$1<p_E\leq q_E<\infty$. For any $y \in E(\cM)$ and any increasing sequence of conditional expectations $(\cE_k)_{k\geq 1}$,
\begin{equation}
\label{eqn:DoobMax} \|(\cE_k(y))_{k\geq 1}\|_{E(\cM;l^{\infty})}
\lesssim_E \|y\|_{E(\cM)}.
\end{equation}
If $(y_k)_{k\geq 1}$ is a martingale in $E(\cM)$, then
$$\sup_{k\geq 1}\|y_k\|_{E(\cM)} \leq \|(y_k)_{k\geq 1}\|_{E(\cM;l^{\infty})} \lesssim_E \sup_{k\geq 1}\|y_k\|_{E(\cM)}.$$
\end{theorem}
\begin{proof}
The first statement follows immediately from Theorem~\ref{thm:linftyInt} and (\ref{eqn:DoobMaxLp}). To prove the second statement, let $y_k = az_kb$ with $(z_k)_{k\geq
1}$ a bounded sequence in $\cM$ and $a,b \in E^{(2)}(\cM)$. Since $E$ is the K\"{o}the dual of a symmetric space, it has the Fatou property and is hence fully symmetric.
Therefore, $\mu(y_k)\prec\prec \mu(az_k)\mu(b)$ implies that
\begin{eqnarray*}
\|y_k\|_{E(\cM)} & \leq & \|\mu(az_k)\mu(b)\|_{E} \\
& \leq & \|az_k\|_{E^{(2)}(\cM)} \|b\|_{E^{(2)}(\cM)} \\
& \leq & \|a\|_{E^{(2)}(\cM)} \|z_k\|_{\infty} \|b\|_{E^{(2)}(\cM)}.
\end{eqnarray*}
Taking the infimum over all decompositions as above gives
$$\sup_{k\geq 1}\|y_k\|_{E(\cM)} \leq \|(y_k)_{k\geq 1}\|_{E(\cM;l^{\infty})}.$$
For the reverse inequality, observe that $E(\cM) \subset L^p(\cM)+L^q(\cM)$ for some $1<p<p_E\leq q_E<q<\infty$. Let $(y_k)_{k\geq 1}$ be a martingale in $E(\cM)$ with
$\sup_{k\geq 1} \|y_k\|_{E(\cM)} = 1$. Then $(y_k)$ is a bounded martingale in $L^p(\cM)+L^q(\cM)$ and hence there exists $y_{\infty} \in L^p(\cM)+L^q(\cM)$ such that
$y_k \rightarrow y_{\infty}$ in $L^p(\cM)+L^q(\cM)$ and $\cE_k(y_{\infty})=y_k$ for all $k\geq 1$. Since $E(\cM)$ has the Fatou property, its unit ball is closed in
$S(\tr)$ (cf.\ \cite{DDP93}, Proposition 5.14). As $y_k\rightarrow y_{\infty}$ in measure, we conclude that $y_{\infty} \in E(\cM)$ and $\|y_{\infty}\|_{E(\cM)}\leq 1$.
Applying (\ref{eqn:DoobMax}) for $y=y_{\infty}$ yields the result.
\end{proof}
\begin{remark}
While the author was completing this paper, a preprint \cite{BCO12} appeared in which Theorems~\ref{thm:maxErgod} and \ref{thm:DoobMax} are proved (under slightly different assumptions) by an alternative method.
\end{remark}

\section{Burkholder-Davis-Gundy and Burkholder-Rosenthal inequalities}

As applications of the noncommutative version of Doob's maximal
inequality and its dual version, we derive versions of the
Burkholder-Davis-Gundy inequalities and Burkholder-Rosenthal
inequalities in noncommutative symmetric spaces. Let $E$ be a
symmetric Banach function space on $\R_+$. For any finite martingale
difference sequence $(x_k)$ in $E(\cM)$ we set
$$\|(x_k)\|_{H^E_c} = \Big\|\Big(\sum_k|x_k|^2\Big)^{\frac{1}{2}}\Big\|_{E(\cM)}; \ \|(x_k)\|_{H^E_r} = \Big\|\Big(\sum_k|x_k^*|^2\Big)^{\frac{1}{2}}\Big\|_{E(\cM)}.$$
These expressions define two norms on the linear space of all finite martingale difference sequences in $E(\cM)$. The following Burkholder-Davis-Gundy inequalities
extend the Burkholder-Gundy inequalities established in \cite{DPP11}.
\begin{theorem}
\label{thm:BDG} Let $E$ be the K\"{o}the dual of a separable symmetric Banach function space on $\R_+$ and suppose that $1<p_E\leq q_E<\infty$. Let $\cM$ be a
semi-finite von Neumann algebra and $(\cM_k)_{k\geq 1}$ a filtration in $\cM$. Then, for any martingale difference sequence $(x_k)$ of a finite martingale $(y_k)$ in $E(\cM)$ we have
$$\|(x_k)_{k\geq 1}\|_{H^E_c+H^E_r} \lesssim_E \|(y_k)_{k\geq 1}\|_{E(\cM;l^{\infty})} \lesssim_E \|(x_k)_{k\geq 1}\|_{H^E_c\cap H^E_r}.$$
Suppose that, moreover, $E$ is separable. If $p_E>1$ and either $q_E<2$ or $E$ is $2$-concave, then
$$\|(y_k)_{k\geq 1}\|_{E(\cM;l^{\infty})} \simeq_E \|(x_k)\|_{H^E_c+H^E_r}.$$
On the other hand, if either $E$ is $2$-convex and $q_E<\infty$ or $2<p_E\leq q_E<\infty$ then
\begin{equation}
\label{eqn:BDG2pqInfty}
\|(y_k)_{k\geq 1}\|_{E(\cM;l^{\infty})} \simeq_E \|(x_k)\|_{H^E_c\cap H^E_r}.
\end{equation}
\end{theorem}
\begin{proof}
If $F$ is a symmetric Banach function space with $F^{\times} = E$, then by (\ref{eqn:dualityBoydIndices}) $1<p_F\leq q_F<\infty$. By Theorem~\ref{thm:DoobMax},
$$\Big\|\sum_{k\geq 1} x_k\Big\|_{E(\cM)} \simeq_E \Big\|\Big(\sum_{k=1}^n x_k\Big)_{n\geq 1}\Big\|_{E(\cM;l^{\infty})}.$$
The result now follows directly from \cite{DPP11}, Proposition 4.18.
\end{proof}
The following result generalizes the noncommutative Rosenthal inequalities presented in \cite{DPP11}, as well as the Burkholder-Rosenthal inequalities for noncommutative $L^p$-spaces and Lorentz spaces obtained in \cite{JuX03}, Theorem 5.1, and \cite{Jia10}, Theorem 3.1, respectively. In fact, Theorem~\ref{thm:BurkRos} positively answers an open question posed in \cite{Jia12}, Problem 3.5(2). The proof follows the general strategy of the proof of \cite{DPP11}, Theorem 6.3. Let $M_n(\cM)$ denote the von Neumann algebra of $n\ti n$ matrices with entries in $\cM$ and for any sequence $(x_k)_{k=1}^n$ in $E(\cM)$ we let $\mathrm{diag}(x_k)$ and $\mathrm{col}(x_k)$ be the matrices with the $x_k$ on the diagonal and first column, respectively, and zeroes elsewhere.
\begin{theorem}
\label{thm:BurkRos} (Noncommutative Burkholder-Rosenthal inequalities) Let $\cM$ be a semi-finite von Neumann algebra. Suppose that $E$ is a symmetric Banach function space on $\R_+$ satisfying $2<p_E\leq q_E<\infty$. Let $(\cM_k)$ be a filtration in $\cM$ and, for every $k\geq 1$, let $\cE_k$ denote the conditional expectation with respect to $\cM_k$. Let $(x_k)$ be a martingale difference sequence in $E(\cM)$ with respect to $(\cM_k)$. Then, for
any $n\geq 1$,
\begin{eqnarray}
\label{eqn:BurkRos}
\Big\|\sum_{k=1}^n x_k\Big\|_{E(\cM)} & \simeq_E & \max\Big\{\|\mathrm{diag}(x_k)_{k=1}^n\|_{E(M_n(\cM))}, \Big\|\Big(\sum_{k=1}^n\cE_{k-1}|x_k|^2\Big)^{\frac{1}{2}}\Big\|_{E(\cM)}, \nonumber \\
& & \ \ \ \ \ \Big\|\Big(\sum_{k=1}^n \cE_{k-1}|x_k^*|^2\Big)^{\frac{1}{2}}\Big\|_{E(\cM)}\Big\}.
\end{eqnarray}
\end{theorem}
\begin{proof}
We first prove that the maximum on the right hand side is dominated by $\|\sum_k x_k\|_{E(\cM)}$. Recall that $L^q(\cM)$ has cotype $q$ if $2\leq q<\infty$, i.e.,
$$\|\mathrm{diag}(x_k)_{k=1}^n\|_{L^q(M_n(\cM))} = \Big(\sum_{k=1}^n \|x_k\|_{L^q(\cM)}^q\Big)^{\frac{1}{q}} \leq \Big\|\sum_{k=1}^n r_k\ot x_k\Big\|_{L^q(L^{\infty}\vNT\cM)}.$$
By interpolating this estimate for $q=2$ and $q>q_E$ we obtain
$$\|\mathrm{diag}(x_k)_{k=1}^n\|_{E(M_n(\cM))} \lesssim_E \Big\|\sum_{k=1}^n r_k\ot x_k\Big\|_{E(L^{\infty}\vNT\cM)}.$$
Moreover, by \cite{DPP11}, Lemma 4.17,
$$\Big\|\sum_k r_k\ot x_k\Big\|_{E(L^{\infty}\vNT\cM)}\simeq_E\Big\|\sum_k x_k\Big\|_{E(\cM)}.$$
Since $1<p_{E_{(2)}}\leq q_{E_{(2)}}<\infty$, we obtain by applying the noncommutative dual Doob inequality (Corollary~\ref{cor:dualDoobSymm}) in $E_{(2)}(\cM)$,
\begin{eqnarray*}
\Big\|\Big(\sum_k\cE_{k-1}(x_k^*x_k)\Big)^{\frac{1}{2}}\Big\|_{E(\cM)} & = & \Big\|\sum_k\cE_{k-1}(x_k^*x_k)\Big\|_{E_{(2)}(\cM)}^{\frac{1}{2}} \\
& \lesssim_E & \Big\|\sum_k x_k^*x_k\Big\|_{E_{(2)}(\cM)}^{\frac{1}{2}} = \Big\|\Big(\sum_k x_k^*x_k\Big)^{\frac{1}{2}}\Big\|_{E(\cM)}.\\
\end{eqnarray*}
Therefore, by the noncommutative Burkholder-Gundy inequality (\cite{DPP11}, Proposition 4.18) we conclude that
$$\Big\|\Big(\sum_k\cE_{k-1}(x_k^*x_k)\Big)^{\frac{1}{2}}\Big\|_{E(\cM)} \lesssim_E \Big\|\Big(\sum_k x_k^*x_k\Big)^{\frac{1}{2}}\Big\|_{E(\cM)} \lesssim_E \Big\|\sum_k x_k\Big\|_{E(\cM)}$$
and by applying this to the sequence $(x_k^*)$ we get
$$\Big\|\Big(\sum_k\cE_{k-1}(x_k x_k^*)\Big)^{\frac{1}{2}}\Big\|_{E(\cM)} \lesssim_E \Big\|\sum_k x_k\Big\|_{E(\cM)}.$$
We now prove the reverse inequality in (\ref{eqn:BurkRos}). By the
noncommutative Burkholder-Gundy inequality,
\begin{equation}
\label{eqn:BRrandKhin}
\Big\|\sum_k x_k\Big\|_{E(\cM)} \lesssim_E\max\Big\{\Big\|\Big(\sum_k x_k^*x_k\Big)^{\frac{1}{2}}\Big\|_{E(\cM)}, \Big\|\Big(\sum_k x_k x_k^*\Big)^{\frac{1}{2}}\Big\|_{E(\cM)}\Big\}.
\end{equation}
By the quasi-triangle inequality in $E_{(2)}(\cM)$ we have
\begin{align}
\label{eqn:firstEstBurkRosSymm}
& \Big\|\Big(\sum_k x_k^*x_k\Big)^{\frac{1}{2}}\Big\|_{E(\cM)} \nonumber \\
& \ \ \lesssim_E \Big(\Big\|\sum_k x_k^*x_k-\cE_{k-1}(x_k^*x_k)\Big\|_{E_{(2)}(\cM)} + \Big\|\sum_k \cE_{k-1}(x_k^*x_k)\Big\|_{E_{(2)}(\cM)}\Big)^{\frac{1}{2}}.
\end{align}
Notice that $(|x_k|^2-\cE_{k-1}(|x_k|^2))_{k\geq 1}$ is a martingale difference sequence in $E_{(2)}(\cM)$. Since $1<p_{E_{(2)}}\leq q_{E_{(2)}}<\infty$ we find by the noncommutative Burkholder-Gundy inequality
\begin{align*}
\Big\|\sum_k x_k^*x_k-\cE_{k-1}(x_k^*x_k)\Big\|_{E_{(2)}(\cM)} \lesssim_E \Big\|\Big(\sum_k (x_k^*x_k-\cE_{k-1}(x_k^*x_k))^{2}\Big)^{\frac{1}{2}}\Big\|_{E_{(2)}(\cM)} \\
\lesssim_E \Big\|\Big(\sum_k|x_k|^4\Big)^{\frac{1}{2}}\Big\|_{E_{(2)}(\cM)} + \Big\|\Big(\sum_k(\cE_{k-1}|x_k|^2)^2\Big)^{\frac{1}{2}}\Big\|_{E_{(2)}(\cM)},
\end{align*}
where in the final inequality we use the quasi-triangle inequality
in $E_{(2)}(\cM;l^2_c)$. By applying the noncommutative Stein
inequality (Corollary~\ref{cor:SteinNC}) to the second term on the
right-hand side, we find that
$$\Big\|\sum_k x_k^*x_k-\cE_{k-1}(x_k^*x_k)\Big\|_{E_{(2)}(\cM)} \lesssim_E \Big\|\Big(\sum_k|x_k|^4\Big)^{\frac{1}{2}}\Big\|_{E_{(2)}(\cM)}.$$
Let $x=\mathrm{col}(|x_k|)$ and $y=\mathrm{diag}(|x_k|)$. Since $\mu(yx)\prec\prec\mu(x)\mu(y)$, it follows from the Calder\'{o}n-Mitjagin theorem (\cite{KPS82}, Theorem
II.3.4) that there is a contraction $T$ for the couple $(L^1,L^{\infty})$ such that $\mu(yx)=T(\mu(x)\mu(y))$. Therefore,
\begin{eqnarray}
\label{eqn:sumFourthPowEst}
\Big\|\Big(\sum_k|x_k|^4\Big)^{\frac{1}{2}}\Big\|_{E_{(2)}(\cM)} & = & \|(x^*y^*yx)^{\frac{1}{2}}\|_{E_{(2)}(M_{n}(\cM))} \nonumber\\
& = & \|y x\|_{E_{(2)}(M_{n}(\cM))} \lesssim_E \|\mu(x)\mu(y)\|_{E_{(2)}} \nonumber\\
& = & \|\mu(x)^{\frac{1}{2}}\mu(y)^{\frac{1}{2}}\|_{E}^2 \leq \|y\|_{E(M_{n}(\cM))}\|x\|_{E(M_{n}(\cM))} \nonumber\\
& = & \|\mathrm{diag}(x_k)\|_{E(M_{n}(\cM))} \Big\|\Big(\sum_k|x_k|^2\Big)^{\frac{1}{2}}\Big\|_{E(\cM)},
\end{eqnarray}
where in the final inequality we use H\"{o}lder's inequality. Putting our estimates together, starting from (\ref{eqn:firstEstBurkRosSymm}), we arrive at
\begin{align*}
& \Big\|\Big(\sum_k|x_k|^2\Big)^{\frac{1}{2}}\Big\|_{E(\cM)} \\
& \ \ \lesssim_E \Big(\|\mathrm{diag}(x_k)\|_{E(M_n(\cM))} \Big\|\Big(\sum_k|x_k|^2\Big)^{\frac{1}{2}}\Big\|_{E(\cM)} + \Big\|\Big(\sum_k
\cE_{k-1}|x_k|^2\Big)^{\frac{1}{2}}\Big\|_{E(\cM)}^2\Big)^{\frac{1}{2}}.
\end{align*}
In other words, if we set $a=\|(\sum_k|x_k|^2)^{\frac{1}{2}}\|_{E(\cM)}$, $b=\|\mathrm{diag}(x_k)\|_{E(M_{n}(\cM))}$ and $c=\|(\sum_k
\cE_{k-1}|x_k|^2)^{\frac{1}{2}}\|_{E(\cM)}$, we have $a^2\lesssim_E ab+c^2$. Solving this quadratic equation we obtain $a\lesssim_E \max\{b,c\}$, or,
\begin{equation*}
\Big\|\Big(\sum_k|x_k|^2\Big)^{\frac{1}{2}}\Big\|_{E(\cM)} \lesssim_E \max\Big\{\|\mathrm{diag}(x_k)\|_{E(M_{n}(\cM))}, \Big\|\Big(\sum_k
\cE_{k-1}|x_k|^2\Big)^{\frac{1}{2}}\Big\|_{E(\cM)}\Big\}.
\end{equation*}
Applying this to the sequence $(x_k^*)$ gives
\begin{equation*}
\Big\|\Big(\sum_k|x_k^*|^2\Big)^{\frac{1}{2}}\Big\|_{E(\cM)} \lesssim_E \max\Big\{\|\mathrm{diag}(x_k)\|_{E(M_{n}(\cM))}, \Big\|\Big(\sum_k
\cE_{k-1}|x_k^*|^2\Big)^{\frac{1}{2}}\Big\|_{E(\cM)}\Big\}.
\end{equation*}
The result now follows by (\ref{eqn:BRrandKhin}).
\epf
\end{proof}
\begin{remark}
Of course, if $E$ is the K\"{o}the dual of a separable symmetric Banach function space on $\R_+$, then by virtue of Theorem~\ref{thm:DoobMax} we may replace $\|\sum_{k=1}^n x_k\|_{E(\cM)}$ by $\|(y_k)_{k=1}^n\|_{E(\cM;l^{\infty})}$ in (\ref{eqn:BurkRos}), where $y_k=\sum_{i=1}^k x_i$.
\end{remark}

\appendix

\section{Original approach to Boyd's theorem}

\noindent In this appendix we prove a full noncommutative analogue of Boyd's interpolation theorem, which allows for interpolation of operators of weak type, instead of
only operators of Marcinkiewicz weak type. We adapt the original proof of Boyd \cite{Boy69}.\par The first result is an extension of Calder\'{o}n's characterization of
weak type operators (\cite{Cal66}, Theorem 8). For $0<p<q<\infty$ define Calder\'{o}n's operator by
$$S_{p,q}f(t) = t^{-\frac{1}{p}}\int_0^t s^{\frac{1}{p}}f(s) \frac{ds}{s} + t^{-\frac{1}{q}} \int_t^{\infty}s^{\frac{1}{q}}f(s) \frac{ds}{s} \ \qquad (t>0, \ f\in S(\R_+))$$
and set
$$S_{p,\infty}f(t) = t^{-\frac{1}{p}}\int_0^t s^{\frac{1}{p}}f(s) \frac{ds}{s} \ \qquad (t>0, \ f\in S(\R_+)).$$
\begin{theorem}
\label{thm:CalChar} Let $\cM,\cN$ be von Neumann algebras equipped with normal, semi-finite, faithful traces $\tr$ and $\si$, respectively. Let $0<p<q\leq \infty$ and
suppose that $T:L^{p,1}(\cM)_+ + L^{q,1}(\cM)_+\rightarrow S(\si)_+$ is a midpoint convex map satisfying
\begin{equation}
\label{eqn:CalWeakEst} \|Tx\|_{L^{r,\infty}(\cN)}\leq C_r \|x\|_{L^{r,1}(\cM)} \qquad (x \in L^{r,1}(\cM)_+, \ r=p,q).
\end{equation}
Then,
$$\mu_t(Tx)\lesssim_{p,q} \max\{C_p,C_q\} \ S_{p,q}\mu(x)(t) \qquad (t>0).$$
\end{theorem}
\begin{proof}
Suppose that $q<\infty$. We may assume that $\max\{C_p,C_q\}\leq 1$. Fix $t>0$, let $x \in L^{p,1}(\cM)_+ + L^{q,1}(\cM)_+$ and set $\delta=\mu_t(x)$. Define $x_t^1 = (x-\delta)e^{x}(\delta,\infty)$ and $x_t^2 = x - x_t^1$, so $x_t^2 =
xe^x[0,\delta]+\delta e^x(\delta,\infty)$. Observe that $x_t^1,x_t^2\geq 0$ as $x\geq 0$. Define two increasing, continuous functions $\phi_1,\phi_2:\R_+\rightarrow \R_+$ by
$$\phi_1(u) = (u-\delta)\chi_{\{u>\delta\}}, \qquad \phi_2(u) = \delta\chi_{\{u>\delta\}} + u\chi_{\{u\leq\delta\}}.$$
From proposition~\ref{pro:DistProperties} it follows that
\begin{align*}
\mu(x_t^1) & = \mu(\phi_1(x)) = \phi_1(\mu(x)) = (\mu(x) - \delta)\chi_{\{\mu(x)\geq\delta\}};\\
\mu(x_t^2) & = \mu(\phi_2(x)) = \phi_2(\mu(x)) = \delta\chi_{\{\mu(x)>\delta\}} + \mu(x)\chi_{\{\mu(x)\leq\delta\}}.
\end{align*}
Using that $\delta=\mu_t(x)$ we obtain
\begin{align*}
\mu_s(x_t^1) & = (\mu_s(x) - \delta)\chi_{[0,t]}(s) \qquad (s>0);\\
\mu_s(x_t^2) & = \delta\chi_{[0,t]}(s) + \mu_s(x)\chi_{(t,\infty)}(s) \qquad (s>0).
\end{align*}
In particular, $\mu(x) = \mu(x_t^1) + \mu(x_t^2)$. By midpoint convexity,
\begin{eqnarray*}
\mu_t(Tx) & \leq & \mu_t(\tfrac{1}{2}T(2x_t^1) + \tfrac{1}{2}T(2x_t^2)) \\
& \leq & \tfrac{1}{2}\mu_{\tfrac{t}{2}}(T(2x_t^1)) + \tfrac{1}{2}\mu_{\frac{t}{2}}(T(2x_t^2)).
\end{eqnarray*}
By (\ref{eqn:CalWeakEst}) we obtain
\begin{equation*}
\mu_{\frac{t}{2}}(T(2x_t^1)) \lesssim_p t^{-\frac{1}{p}}\|x_t^1\|_{L^{p,1}(\cM)} = t^{-\frac{1}{p}}\int_0^t s^{\frac{1}{p}}\mu_s(x_t^1) \frac{ds}{s} =
S_{p,q}\mu(x_t^1)(t)
\end{equation*}
and, moreover,
\begin{eqnarray*}
\mu_{\frac{t}{2}}(T(2x_t^2)) & \lesssim_q & t^{-\frac{1}{q}}\|x_t^2\|_{L^{q,1}(\cM)}\\
& = & t^{-\frac{1}{q}} \int_0^{\infty}s^{\frac{1}{q}}\mu_s(x_t^2) \frac{ds}{s} \\
& = & t^{-\frac{1}{q}} \int_0^{t}s^{\frac{1}{q}}\delta \frac{ds}{s} + t^{-\frac{1}{q}} \int_t^{\infty}s^{\frac{1}{q}}\mu_s(x) \frac{ds}{s}.
\end{eqnarray*}
Since
$$t^{-\frac{1}{q}} \int_0^{t}s^{\frac{1}{q}}\delta \frac{ds}{s} = q\delta = \frac{q}{p} t^{-\frac{1}{p}} \int_0^{t}s^{\frac{1}{p}}\delta \frac{ds}{s},$$
it follows that
$$\mu_{\frac{t}{2}}(T(2x_t^2)) \lesssim_{p,q} S_{p,q}\mu(x_t^2)(t).$$
Putting our estimates together, we conclude that
$$\mu_t(Tx)\lesssim_{p,q} S_{p,q}\mu(x_t^1)(t) + S_{p,q}\mu(x_t^2)(t) = S_{p,q}\mu(x)(t).$$
The case $q=\infty$ is proved similarly.
\end{proof}
The following observation for symmetric Banach function spaces and $p,q\geq 1$ is the main result of \cite{Boy69}.
The general case is proved using essentially the same argument (see \cite{Mon96}, Theorem 2).
\begin{theorem}
\label{thm:BoydObsOrig} If $E$ is a symmetric quasi-Banach function space on $\R_+$, then the following hold.
\begin{enumerate}
\item If $0<p<q<\infty$, then $S_{p,q}$ is bounded on $E$ if and only if $p<p_E\leq q_E<q$.
\item If $0<p<\infty$, then $S_{p,\infty}$ is bounded on $E$ if and only if $p<p_E$.
\end{enumerate}
\end{theorem}
By combining the observations in Theorems~\ref{thm:CalChar} and \ref{thm:BoydObsOrig} we find the following noncommutative extension of Boyd's theorem.
\begin{theorem}
\label{thm:BoydOrigApp} Let $E$ be a symmetric quasi-Banach function space on $\R_+$ which is $s$-convex for some $0<s<\infty$. Let $\cM,\cN$ be von Neumann algebras
equipped with normal, semi-finite, faithful traces $\tr$ and $\si$, respectively. Suppose that $0<p<q\leq\infty$ and let $T:L^{p,1}(\cM)_+ + L^{q,1}(\cM)_+\rightarrow
S(\si)_+$ be a midpoint convex map such that
\begin{equation*}
\|Tx\|_{L^{r,\infty}(\cN)} \leq C_r \|x\|_{L^{r,1}(\cM)} \qquad (x \in L^{r,1}(\cM)_+, \ r=p,q).
\end{equation*}
If $p<p_E\leq q_E<q<\infty$ or $p<p_E$ and $q=\infty$, then
$$\|Tx\|_{E(\cN)} \lesssim_{p,q} \max\{C_p,C_q\} \|S_{p,q}\|_{E\rightarrow E} \ \|x\|_{E(\cM)} \qquad (x \in E(\cM)_+).$$
\end{theorem}

\section*{Acknowledgement}

\noindent I would like to thank the anonymous reviewer for his/her comments on an earlier version of this paper and the resulting improvements.

\bibliographystyle{plain}
\bibliography{Boyd_Accepted}

\begin{thebibliography}{10}

\bibitem{BeC10}
T.~Bekjan and Z.~Chen.
\newblock Interpolation and {$\Phi$}-moment inequalities of noncommutative
  martingales.
\newblock {\em Probab. Theory Related Fields}, 152:179--206.

\bibitem{BCO12}
T.~Bekjan, Z.~Chen, and A.~Os{\c{e}}kowski.
\newblock Noncommutative moment maximal inequalities.
\newblock ArXiv: 1108.2795.

\bibitem{BCP62}
A.~Benedek, A.-P. Calder{\'o}n, and R.~Panzone.
\newblock Convolution operators on {B}anach space valued functions.
\newblock {\em Proc. Nat. Acad. Sci. U.S.A.}, 48:356--365, 1962.

\bibitem{BeS88}
C.~Bennett and R.~Sharpley.
\newblock {\em Interpolation of operators}.
\newblock Academic Press Inc., Boston, MA, 1988.

\bibitem{Boy69}
D.~Boyd.
\newblock Indices of function spaces and their relationship to interpolation.
\newblock {\em Canad. J. Math.}, 21:1245--1254, 1969.

\bibitem{Cal66}
A.-P. Calder{\'o}n.
\newblock Spaces between {$L\sp{1}$} and {$L\sp{\infty }$} and the theorem of
  {M}arcinkiewicz.
\newblock {\em Studia Math.}, 26:273--299, 1966.

\bibitem{DPP11}
S.~Dirksen, B.~de~Pagter, D.~Potapov, and F.~Sukochev.
\newblock Rosenthal inequalities in noncommutative symmetric spaces.
\newblock {\em J. Funct. Anal.}, 261(10):2890--2925, 2011.

\bibitem{DiR11}
S.~Dirksen and \'E. Ricard.
\newblock Some remarks on noncommutative {K}hintchine inequalities.
\newblock {\em To appear in Bull. London Math. Soc.}
\newblock ArXiv: 1108.5332.

\bibitem{DDP89}
P.~Dodds, T.~Dodds, and B.~de~Pagter.
\newblock Noncommutative {B}anach function spaces.
\newblock {\em Math. Z.}, 201(4):583--597, 1989.

\bibitem{DDP92}
P.~Dodds, T.~Dodds, and B.~de~Pagter.
\newblock Fully symmetric operator spaces.
\newblock {\em Integral Equations Operator Theory}, 15(6):942--972, 1992.

\bibitem{DDP93}
P.~Dodds, T.~Dodds, and B.~de~Pagter.
\newblock Noncommutative {K}\"othe duality.
\newblock {\em Trans. Amer. Math. Soc.}, 339(2):717--750, 1993.

\bibitem{DPS11}
P.~Dodds, B.~de Pagter, and F.~Sukochev.
\newblock Noncommutative integration.
\newblock Work in progress.

\bibitem{FaK86}
T.~Fack and H.~Kosaki.
\newblock Generalized {$s$}-numbers of {$\tau$}-measurable operators.
\newblock {\em Pacific J. Math.}, 123(2):269--300, 1986.

\bibitem{Hu09}
Y.~Hu.
\newblock Noncommutative extrapolation theorems and applications.
\newblock {\em Illinois J. Math.}, 53(2):463--482, 2009.

\bibitem{Jia10}
Y.~Jiao.
\newblock Burkholder's inequalities in noncommutative {L}orentz spaces.
\newblock {\em Proc. Amer. Math. Soc.}, 138(7):2431--2441, 2010.

\bibitem{Jia12}
Y.~Jiao.
\newblock Martingale inequalities in noncommutative symmetric spaces.
\newblock {\em Arch. Math. (Basel)}, 98(1):87--97, 2012.

\bibitem{Jun02}
M.~Junge.
\newblock Doob's inequality for non-commutative martingales.
\newblock {\em J. Reine Angew. Math.}, 549:149--190, 2002.

\bibitem{JuX03}
M.~Junge and Q.~Xu.
\newblock Noncommutative {B}urkholder/{R}osenthal inequalities.
\newblock {\em Ann. Probab.}, 31(2):948--995, 2003.

\bibitem{JuX07}
M.~Junge and Q.~Xu.
\newblock Noncommutative maximal ergodic theorems.
\newblock {\em J. Amer. Math. Soc.}, 20(2):385--439, 2007.

\bibitem{Kal84}
N.~Kalton.
\newblock Compact and strictly singular operators on certain function spaces.
\newblock {\em Arch. Math. (Basel)}, 43(1):66--78, 1984.

\bibitem{KPR84}
N.~Kalton, N.~Peck, and J.~Roberts.
\newblock {\em An {$F$}-space sampler}.
\newblock Cambridge University Press, Cambridge, 1984.

\bibitem{KaS08}
N.~Kalton and F.~Sukochev.
\newblock Symmetric norms and spaces of operators.
\newblock {\em J. Reine Angew. Math.}, 621:81--121, 2008.

\bibitem{KPS82}
S.~Kre{\u\i}n, Y.~Petun{\={\i}}n, and E.~Sem{\"e}nov.
\newblock {\em Interpolation of linear operators}.
\newblock American Mathematical Society, Providence, R.I., 1982.

\bibitem{LiT79}
J.~Lindenstrauss and L.~Tzafriri.
\newblock {\em Classical {B}anach spaces. {II}}.
\newblock Springer-Verlag, Berlin, 1979.

\bibitem{L-P86}
F.~Lust-Piquard.
\newblock In\'egalit\'es de {K}hintchine dans {$C\sb p\;(1<p<\infty)$}.
\newblock {\em C. R. Acad. Sci. Paris S\'er. I Math.}, 303(7):289--292, 1986.

\bibitem{LuP91}
F.~Lust-Piquard and G.~Pisier.
\newblock Noncommutative {K}hintchine and {P}aley inequalities.
\newblock {\em Ark. Mat.}, 29(2):241--260, 1991.

\bibitem{Mal85}
L.~Maligranda.
\newblock Indices and interpolation.
\newblock {\em Dissertationes Math. (Rozprawy Mat.)}, 234:49, 1985.

\bibitem{Mar39}
J.~Marcinkiewicz.
\newblock {Sur l'interpolation d'op\'erations.}
\newblock {\em C. R. Acad. Sci., Paris}, 208:1272--1273, 1939.

\bibitem{Mon96}
S.~Montgomery-Smith.
\newblock The {H}ardy operator and {B}oyd indices.
\newblock In {\em Interaction between functional analysis, harmonic analysis,
  and probability ({C}olumbia, {MO}, 1994)}, volume 175 of {\em Lecture Notes
  in Pure and Appl. Math.}, pages 359--364. Dekker, New York, 1996.

\bibitem{Nel74}
E.~Nelson.
\newblock Notes on non-commutative integration.
\newblock {\em J. Funct. Anal.}, 15:103--116, 1974.

\bibitem{Pis98}
G.~Pisier.
\newblock Non-commutative vector valued {$L\sb p$}-spaces and completely
  {$p$}-summing maps.
\newblock {\em Ast\'erisque}, (247):vi+131, 1998.

\bibitem{PiX97}
G.~Pisier and Q.~Xu.
\newblock Non-commutative martingale inequalities.
\newblock {\em Comm. Math. Phys.}, 189(3):667--698, 1997.

\bibitem{StW59}
E.~Stein and G.~Weiss.
\newblock An extension of a theorem of {M}arcinkiewicz and some of its
  applications.
\newblock {\em J. Math. Mech.}, 8:263--284, 1959.

\bibitem{Xu07}
Q.~Xu.
\newblock Operator spaces and noncommutative {$L_p$}.
\newblock {N}ankai university summer school lecture notes, 2007.

\bibitem{Xu91}
Q.~Xu.
\newblock Analytic functions with values in lattices and symmetric spaces of
  measurable operators.
\newblock {\em Math. Proc. Cambridge Philos. Soc.}, 109(3):541--563, 1991.

\bibitem{Zyg56}
A.~Zygmund.
\newblock On a theorem of {M}arcinkiewicz concerning interpolation of
  operations.
\newblock {\em J. Math. Pures Appl. (9)}, 35:223--248, 1956.

\end{thebibliography}

\end{document}